\documentclass[11pt, oneside, article]{memoir}

\settrims{0pt}{0pt} % page and stock same size
\settypeblocksize{*}{35pc}{*} % {height}{width}{ratio}
\setlrmargins{*}{*}{1} % {spine}{edge}{ratio}
\setulmarginsandblock{1.1in}{1.4in}{*} % height of typeblock computed
\setheadfoot{\onelineskip}{2\onelineskip} % {headheight}{footskip}
\setheaderspaces{*}{1.5\onelineskip}{*} % {headdrop}{headsep}{ratio}
\checkandfixthelayout

\usepackage[centertags,sumlimits,intlimits,namelimits,reqno]{amsmath}
\usepackage{latexsym}
\usepackage{amsfonts,amsthm,amssymb}
\usepackage{mathtools}
\usepackage[cal=euler,scr=rsfso]{mathalfa}
\usepackage[boxslash]{stmaryrd}
\usepackage{newpxtext}
\usepackage{enumitem}
\usepackage{ifthen}
\usepackage[colorlinks,linkcolor=darkblue,citecolor=darkblue,urlcolor=darkblue,breaklinks=true, backref=true]{hyperref}
\usepackage{tikz}
\usepackage[capitalize]{cleveref}
\usepackage[backend=biber, backref, bibencoding=utf8, maxbibnames = 10, style = alphabetic]{biblatex}
\DeclareMathAlphabet{\mathpzc}{OT1}{pzc}{m}{it}
\usepackage{xstring}
\usepackage{graphicx}

\usepackage[color=white, textsize=footnotesize]{todonotes}

\presetkeys%
    {todonotes}%
    {inline}{}
    
\DeclareFontFamily{U}{mathx}{\hyphenchar\font45}
\DeclareFontShape{U}{mathx}{m}{n}{
      <5> <6> <7> <8> <9> <10>
      <10.95> <12> <14.4> <17.28> <20.74> <24.88>
      mathx10
      }{}
\DeclareSymbolFont{mathx}{U}{mathx}{m}{n}
\DeclareFontSubstitution{U}{mathx}{m}{n}
\DeclareMathAccent{\widecheck}{0}{mathx}{"71}

	\allowdisplaybreaks
	
  \definecolor{darkblue}{rgb}{0,0,0.7}

  \usetikzlibrary{ 
  	cd,
  	math,
  	decorations.markings,
		decorations.pathreplacing,
  	positioning,
	 	circuits.logic.US,
 		arrows.meta,
  	shapes,
		shadows,
		shadings,
  	calc,
  	fit,
  	quotes,
  	intersections,
  }

\newif\ifpgfshaperectangleroundnortheast
\newif\ifpgfshaperectangleroundnorthwest
\newif\ifpgfshaperectangleroundsoutheast
\newif\ifpgfshaperectangleroundsouthwest
\pgfkeys{/pgf/.cd,
  rectangle round north east/.is if=pgfshaperectangleroundnortheast,
  rectangle round north west/.is if=pgfshaperectangleroundnorthwest,
  rectangle round south east/.is if=pgfshaperectangleroundsoutheast,
  rectangle round south west/.is if=pgfshaperectangleroundsouthwest,
  rectangle round north east, rectangle round north west,
  rectangle round south east, rectangle round south west,
}
\makeatletter
\def\pgf@sh@bg@rectangle{%
  \pgfkeysgetvalue{/pgf/outer xsep}{\outerxsep}%
  \pgfkeysgetvalue{/pgf/outer ysep}{\outerysep}%
  \pgfpathmoveto{\pgfpointadd{\southwest}{\pgfpoint{\outerxsep}{\outerysep}}}%
  {\ifpgfshaperectangleroundnorthwest\else\pgfsetcornersarced{\pgfpointorigin}\fi%
    \pgfpathlineto{\pgfpointadd{\southwest\pgf@xa=\pgf@x\northeast\pgf@x=\pgf@xa}{\pgfpoint{\outerxsep}{-\outerysep}}}}%  
  {\ifpgfshaperectangleroundnortheast\else\pgfsetcornersarced{\pgfpointorigin}\fi%  
    \pgfpathlineto{\pgfpointadd{\northeast}{\pgfpoint{-\outerxsep}{-\outerysep}}}}%
  {\ifpgfshaperectangleroundsoutheast\else\pgfsetcornersarced{\pgfpointorigin}\fi%  
    \pgfpathlineto{\pgfpointadd{\southwest\pgf@ya=\pgf@y\northeast\pgf@y=\pgf@ya}{\pgfpoint{-\outerxsep}{\outerysep}}}}%
  {\ifpgfshaperectangleroundsouthwest\else\pgfsetcornersarced{\pgfpointorigin}\fi%  
    \pgfpathclose}}

\tikzset{
  	WD/.style={%everything after equals replaces "oriented WD" in key.
  	  label/.style={
      	font=\everymath\expandafter{\the\everymath\scriptstyle},
        inner sep=0pt,
        node distance=2pt and -2pt},
      semithick,
      node distance=\bbx and \bby,
      decoration={markings, mark=at position \stringdecpos with \stringdec},
    	bb port length=0,
  	  bb port sep=.5,
	 	  bbx = .4cm,
		  bb min width=.4cm,
	    bby = 2ex,
	    bb penetrate=0,
	    bb rounded corners=2pt,
	    dot size=3pt,
      pack size = 16pt,
    	penetration = 0pt,
      link size = 2pt,
      pack color = blue,
      surround sep=2pt,
      ar/.style={postaction={decorate}},
  		execute at begin picture={\tikzset{
  			x=\bbx, y=\bby, 
				circuit logic US, tiny circuit symbols
				}
			}
    },
    bbx/.store in=\bbx,
    bby/.store in=\bby,
    bb port sep/.store in=\bbportsep,
    bb port length/.store in=\bbportlen,
    bb penetrate/.store in=\bbpenetrate,
    bb min width/.store in=\bbminwidth,
    bb rounded corners/.store in=\bbcorners,
    bb/.code 2 args={%When you see this key, run the code below:
    	\pgfmathsetlengthmacro{\bbheight}{\bbportsep * (max(#1,#2)+1) * \bby}
      \pgfkeysalso{draw,minimum height=\bbheight,minimum
       width=\bbminwidth,outer sep=0pt,
         rounded corners=\bbcorners,thick,
         prefix after command={\pgfextra{\let\fixname\tikzlastnode}},
         append after command={\pgfextra{\draw
            \ifnum #1=0{} \else foreach \i in {1,...,#1} {
            	($(\fixname.north west)!{(2*\i-1)/(2*#1)}!(\fixname.south west)$) +(-\bbportlen,0) coordinate (\fixname_in\i) -- +(\bbpenetrate,0) coordinate (\fixname_in\i')}\fi 
            \ifnum #2=0{} \else foreach \i in {1,...,#2} {
            	($(\fixname.north east)!{(2*\i-1)/(2*#2)}!(\fixname.south east)$) +(-
\bbpenetrate,0) coordinate (\fixname_out\i') -- +(\bbportlen,0) coordinate (\fixname_out\i)}\fi;
           }}}
		},
	dot size/.store in=\dotsize,
	dot/.style={
		circle, draw, thick, inner sep=0, fill=black, minimum width=\dotsize
	},
	pack size/.store in=\psize,
	penetration/.store in=\penetration,
  spacing/.store in=\spacing,
  link size/.store in=\lsize,
  pack color/.store in=\pcolor,
 	pack inside color/.store in=\picolor,
  pack inside color=blue!20,
 	pack outside color/.store in=\pocolor,
  pack outside color=blue!50!black,
 	surround sep/.store in=\ssep,
 	link/.style={
  	circle, 
  	draw=black, 
  	fill=black,
  	inner sep=0pt, 
 		minimum size=\lsize
 	},
  pack/.style={
 		circle, 
 		draw = \pocolor, 
  	fill = \picolor,
  	minimum size = \psize
  },
  func/.style={
  	pack,
		rectangle,
		rounded corners=.5*\psize,
		inner ysep=.125*\psize,
		minimum width=1.125*\psize,
		inner xsep=.25*\psize,
  },
  funcr/.style={
    func,
    rectangle round north west=false, 
		rectangle round south west=false,
  },
  funcl/.style={
    func,
		rectangle round north east=false, 
		rectangle round south east=false,
  },
  funcu/.style={
    func,
		rectangle round south east=false, 
		rectangle round south west=false,
  },
  funcd/.style={
    func,
		rectangle round north east=false, 
		rectangle round north west=false,
  },
  outer pack/.style={
 		ellipse, 
 		draw,
  	inner sep=\ssep,
  	color=gray,
 	},
  intermediate pack/.style={
 		ellipse,
 		dashed, 
  	draw,
  	inner sep=\ssep,
 		color=\pocolor,
 	},
 }
 
\tikzset{light gray nodes/.style={every node/.style={fill=gray!40}}}

\tikzset{
	oriented WD/.style={%everything after equals replaces "oriented WD" in key.
		every to/.style={out=0,in=180,draw},
    label/.style={
    	font=\everymath\expandafter{\the\everymath\scriptstyle},
      inner sep=0pt,
      node distance=2pt and -2pt},
    semithick,
    node distance=1 and 1,
    decoration={markings, mark=at position \stringdecpos with \stringdec},
    ar/.style={postaction={decorate}},
    execute at begin picture={\tikzset{
    	x=\bbx, y=\bby,
      every fit/.style={inner xsep=\bbx, inner ysep=\bby}}}
    },
    string decoration/.store in=\stringdec,
    string decoration={\arrow{stealth};},
    string decoration pos/.store in=\stringdecpos,
    string decoration pos=.7,
    bbx/.store in=\bbx,
    bbx = 1.5cm,
    bby/.store in=\bby,
    bby = 1.5ex,
    bb port sep/.store in=\bbportsep,
    bb port sep=1.5,
    bb port length/.store in=\bbportlen,
    bb port length=4pt,
    bb penetrate/.store in=\bbpenetrate,
    bb penetrate=0,
    bb min width/.store in=\bbminwidth,
    bb min width=1cm,
    bb rounded corners/.store in=\bbcorners,
    bb rounded corners=2pt,
    bb spider/.style={
    	bb port sep=1, bb port length=10pt, bbx=.4cm, bb min width=.4cm, bby=.8ex},
    bb small/.style={
    	bb port sep=1, bb port length=2.5pt, bbx=.4cm, bb min width=.4cm, bby=.7ex},
		bb medium/.style={
			bb port sep=1, bb port length=2.5pt, bbx=.4cm, bb min width=.4cm, bby=.9ex},
    bb/.code 2 args={%When you see this key, run the code below:
    	\pgfmathsetlengthmacro{\bbheight}{\bbportsep * (max(#1,#2)+1) * \bby}
      \pgfkeysalso{draw,minimum height=\bbheight,minimum
       width=\bbminwidth,outer sep=0pt,
         rounded corners=\bbcorners,thick,
         prefix after command={\pgfextra{\let\fixname\tikzlastnode}},
         append after command={\pgfextra{\draw
            \ifnum #1=0{} \else foreach \i in {1,...,#1} {
            	($(\fixname.north west)!{(2*\i-1)/(2*#1)}!(\fixname.south west)$) +(-\bbportlen,0) coordinate (\fixname_in\i) -- +(\bbpenetrate,0) coordinate (\fixname_in\i')}\fi 
            \ifnum #2=0{} \else foreach \i in {1,...,#2} {
            	($(\fixname.north east)!{(2*\i-1)/(2*#2)}!(\fixname.south east)$) +(-
\bbpenetrate,0) coordinate (\fixname_out\i') -- +(\bbportlen,0) coordinate (\fixname_out\i)}\fi;
           }}}
		},
		bb name/.style={
    	append after command={
				\pgfextra{\node[anchor=north] at (\fixname.north) {#1};}
			}
		},
  }

\tikzset{
	unoriented WD/.style={
  	every to/.style={draw},
  	shorten <=-\penetration, shorten >=-\penetration,
  	label distance=-2pt,
  	thick,
  	node distance=\spacing,
  	execute at begin picture={\tikzset{
  		x=\spacing, y=\spacing, circuit logic US, tiny circuit symbols}
		}
  },
  pack size/.store in=\psize,
  pack size = 12pt,
	penetration/.store in=\penetration,
	penetration = 0pt,
  spacing/.store in=\spacing,
  spacing = 8pt,
  link size/.store in=\lsize,
  link size = 2pt,
  pack color/.store in=\pcolor,
  pack color = blue,
 	pack inside color/.store in=\picolor,
  pack inside color=blue!20,
 	pack outside color/.store in=\pocolor,
  pack outside color=blue!50!black,
 	surround sep/.store in=\ssep,
  surround sep=8pt,
 	link/.style={
  	circle, 
  	draw=black, 
  	fill=black,
  	inner sep=0pt, 
 		minimum size=\lsize
 	},
  pack/.style={
 		circle, 
 		draw = \pocolor, 
  	fill = \picolor,
  	minimum size = \psize
  },
  func/.style={
  	pack,
		rectangle,
		rounded corners=.5*\psize,
		inner ysep=.125*\psize,
		minimum width=1.125*\psize,
		inner xsep=.25*\psize,
  },
  funcr/.style={
    func,
    rectangle round north west=false, 
		rectangle round south west=false,
  },
  funcl/.style={
    func,
		rectangle round north east=false, 
		rectangle round south east=false,
  },
  funcu/.style={
    func,
		rectangle round south east=false, 
		rectangle round south west=false,
  },
  funcd/.style={
    func,
		rectangle round north east=false, 
		rectangle round north west=false,
  },
  outer pack/.style={
 		ellipse, 
 		draw,
  	inner sep=\ssep,
  	color=gray,
 	},
  intermediate pack/.style={
 		ellipse,
 		dashed, 
  	draw,
  	inner sep=\ssep,
 		color=\pocolor,
 	},
}
  
\tikzset{
	spider diagram/.style={
		every to/.style={out=0, in=180, draw, thick},
		thick,
		dot size = 5pt,
		execute at begin picture={\tikzset{
    	x=\leglen, y=\leglen/3}}
	},
	dot size/.store in=\dotsize,
	dot fill/.store in=\dotfill,
	dot fill = black,
	leg length/.store in=\leglen,
	leg length = 15pt,
	baby/.style={dot size = 2pt, leg length = 6pt},
	young/.style={dot size = 3pt, leg length = 10pt},
	adolescent/.style={dot size = 4pt, leg length = 12pt},
	special spider/.code n args={4}{
		\pgfkeysalso{circle, draw, thick, inner sep=0, fill=\dotfill, minimum width=\dotsize,
  		prefix after command={\pgfextra{\let\fixname\tikzlastnode}},
  		append after command={\pgfextra{
  			\ifnum #1=0{} \else {\foreach \i in {1,...,#1} {
					\tikzmath{\anglei={-90*(#1+1-2*\i)/#1};}
  				\draw [thick]
						(\fixname) .. controls 
						($(\fixname.center)-(\anglei:#3/3)$) and ($(\fixname.center)-(\anglei:#3*2/3)$) .. 
						({$(\fixname.center)-(\anglei:#3*2/3)$}-|{$(\fixname.center)-(#3,0)$}) coordinate (\fixname_in\i);
  			}}\fi
  			\ifnum #2=0{} \else {\foreach \i in {1,...,#2} {
					\tikzmath{\anglei={90*(#2+1-2*\i)/#2};}
  				\draw [thick]
						(\fixname.center) .. controls 
						($(\fixname.center)+(\anglei:#4/3)$) and ($(\fixname.center)+(\anglei:#4*2/3)$) .. 
						({$(\fixname.center)+(\anglei:#4*2/3)$}-|{$(\fixname.center)+(#4,0)$}) coordinate (\fixname_out\i);
  			}}\fi
  		}}
		}
	},
	spider/.code 2 args={
		\pgfkeysalso{special spider={#1}{#2}{\leglen}{\leglen}}
	}
}

\tikzset{
	inner WD/.style={
		every to/.style={out=0, in=180, draw, thick},
		unoriented WD, 
		surround sep=2pt, 
		font=\tiny, 
		anchor=center
	}
}

\tikzset{
  function/.style={->, thin, shorten <=4pt, shorten >=4pt}
}

\tikzset{
  tick/.style={
  	postaction={
    	decorate,
      decoration={
      	markings, mark=at position 0.5 with {
					\draw[-] (0,.4ex) -- (0,-.4ex);
				}
			}
		}
	}
}

\tikzcdset{
	twocell/.style={Rightarrow, shorten <=5pt, shorten >=5pt}
}

\DefineBibliographyStrings{english}{%
  backrefpage = {see page},% originally "cited on page"
  backrefpages = {see pages},% originally "cited on pages"
}

  \setlist{noitemsep, nolistsep}
	\setlist[description]{leftmargin=0em, itemindent=2em}
		
\theoremstyle{plain}
\newtheorem*{theorem*}{Theorem} 
\newtheorem{theorem}{Theorem}[chapter] 
\newtheorem{proposition}[theorem]{Proposition}
\newtheorem*{proposition*}{Proposition}
\newtheorem{corollary}[theorem]{Corollary}
\newtheorem{lemma}[theorem]{Lemma}

\theoremstyle{definition}
\newtheorem{definition}[theorem]{Definition}

\newtheorem*{axiom*}{Axiom}

\theoremstyle{remark}
\newtheorem{example}[theorem]{Example}
\newtheorem{remark}[theorem]{Remark}

\crefalias{chapter}{section}

\newcommand{\Set}[1]{\mathrm{#1}}%a named set
\newcommand{\ord}[1]{\underline{#1}}%a natural number, considered as a finite set
\newcommand{\cat}[1]{\mathcal{#1}}%a generic category
\newcommand{\ccat}[1]{\mathbb{#1}}%a generic category
\newcommand{\Cat}[1]{{\mathsf{#1}}}%a named category
\newcommand{\CCat}[1]{\mathbb{\StrLeft{#1}{1}}\Cat{\StrGobbleLeft{#1}{1}}}%a named category; does not seem to work in section headers...
\DeclareMathOperator{\ob}{\Set{Ob}}
\DeclareMathOperator{\id}{id}

\DeclareMathOperator{\Hom}{Hom}

\DeclareMathOperator{\inc}{inc}
\DeclarePairedDelimiter{\pair}{\langle}{\rangle}

\DeclarePairedDelimiter{\copair}{[}{]}

\newcommand{\tn}[1]{\textnormal{#1}}
\newcommand{\op}{^{\tn{op}}}

\newcommand{\tpow}[1]{^{\otimes #1}}
\newcommand{\strict}[1]{\overline{#1}}

\newcommand{\finset}{\Cat{FinSet}}
\newcommand{\bij}{\bb}
\newcommand{\smset}{\Cat{Set}}

\newcommand{\smf}{\Cat{SMF}}
\newcommand{\ssmc}{\CCat{SMC}}

\newcommand{\bb}{\mathbb{B}} %Not sure what old \aa did
\newcommand{\dd}{\mathbb{D}}
\newcommand{\nn}{\mathbb{N}}
\newcommand{\pp}{\mathbb{P}}
\newcommand{\qq}{\mathbb{Q}}

\newcommand{\mob}[1]{#1_0}

\newcommand{\rel}{\Cat{Rel}}
\newcommand{\cospan}{\Cat{Cospan}}
\newcommand{\zero}{\cat{I}}

\newcommand{\Ldots}[1]{\overset{#1}{\ldots}}
\newcommand{\Cdots}[1]{\overset{#1}{\cdots}}

\newcommand{\cp}{\mathbin{\fatsemi}}

\newcommand{\To}[1]{\xrightarrow{#1}}
\newcommand{\Too}[1]{\To{\;\;#1\;\;}}
\newcommand{\from}{\leftarrow}

\renewcommand{\ss}{\subseteq}

\newcommand{\qqand}{\qquad\text{and}\qquad}

\newcommand{\hide}[2][]{#1}

\begin{document}   

\title{Supplying bells and whistles in\\symmetric monoidal categories}
\author{Brendan Fong \and David I.\ Spivak}
\date{\vspace{-.3in}}
  
\maketitle

%============ Abstract ============%
\begin{abstract}
It is common to encounter symmetric monoidal categories $\cat{C}$ for which every object is equipped with an algebraic structure, in a way that is compatible with the monoidal product and unit in $\cat{C}$. We define this formally and say that $\cat{C}$ \emph{supplies} the algebraic structure. For example, the category $\rel$ of relations between sets has monoidal structures given by both cartesian product and disjoint union, and with respect to either one it supplies comonoids. We prove several facts about the notion of supply, e.g.\ that the associators, unitors, and braiding of $\cat{C}$ are automatically homomorphisms for any supply, as are the coherence isomorphisms for any strong symmetric monoidal functor that preserve supplies. We also show that any supply of structure in a symmetric monoidal category can be extended to a supply of that structure on its strictification.

%Along the way, we also prove that the 2-category of SM categories, strong SM functors, and monoidal natural transformations has biproducts, a fact which does not appear to be well-known and which we had difficulty finding in the literature.
\end{abstract}

%======== Chapter ========%
\chapter{Introduction}

Many symmetric monoidal categories $\cat{C}$ have the property that each object $c\in\cat{C}$ is equipped with a certain algebraic structure---say that of a monoid or a comonoid---in a way that is compatible with $\cat{C}$'s monoidal structure.

For example, consider the category $\rel$ of relations between sets. It has a symmetric monoidal structure $(I, \otimes, \gamma)$ coming from the cartesian monoidal structure of $\smset$. This is not a cartesian monoidal structure on $\rel$: indeed, $I$ is not terminal. And yet each object $r\in\rel$ is equipped with morphisms $\delta_r\colon r\to r\otimes r$ and $\epsilon_r\colon r\to I$, which satisfy the same \emph{algebraic} properties that a diagonal and a terminal morphism do. Namely, the diagrams expressing commutativity, unitality, and associativity commute:
\[
\begin{tikzcd}[column sep=small]
	&r\ar[dl, "\delta"']\ar[dr, "\delta"]\\
	r\otimes r\ar[rr, "\gamma"']&&
	r\otimes r
\end{tikzcd}
\qquad
\begin{tikzcd}[column sep=small]
	r\ar[rr, equal]\ar[dr, "\delta"']&&
	r\\&
	r\otimes r\ar[ur, "\epsilon\otimes r"']
\end{tikzcd}
\qquad
\begin{tikzcd}
	r\ar[r, "\delta"]\ar[d, "\delta"']&
	r\otimes r\ar[d, "r\otimes\delta"]\\
	r\otimes r\ar[r, "\delta\otimes r"']&
	r\otimes r\otimes r
\end{tikzcd}
\]
In string diagrams, the maps $\delta$ and $\epsilon$ can be drawn as:
\[
\begin{tikzpicture}[WD]
	\node[bb={1}{0}] (eta') {$\epsilon$};
	\draw (eta'_in1) to +(-.8,0);
	\node[bb={1}{2}, right=4 of eta'] (mu') {$\delta$};
	\draw (mu'_in1) -- +(-.8,0);
	\draw (mu'_out1) -- +(.8,0);
	\draw (mu'_out2) -- +(.8,0);
\end{tikzpicture}
\]
and the equations can be drawn as:
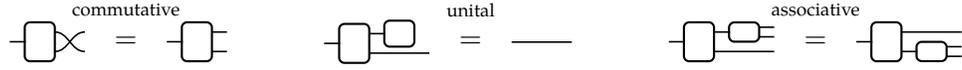
\[
\begin{tikzpicture}
	\node (Q11) {
	\begin{tikzpicture}[WD]
		\node[bb={1}{2}] (a) {};
		\coordinate (a1) at ($(a_out1)+(1,0)$);
		\coordinate (a2) at ($(a_out2)+(1,0)$);
		\draw (a_out2) to[out=0, in=180] (a1);
		\draw (a_out1) to[out=0, in=180] (a2);
		\draw (a_in1) -- +(-.5,0);
	\end{tikzpicture}
	};
	\node (Q12) [right=.8 of Q11] {
	\begin{tikzpicture}[WD]
		\node[bb={1}{2}] (a) {};
		\draw (a_out1) -- +(.5,0);
		\draw (a_out2) -- +(.5,0);
		\draw (a_in1) -- +(-.5,0);
	\end{tikzpicture}
	};
	\node[label=above:{\tiny commutative}] at ($(Q11.east)!.5!(Q12.west)$) {$=$};
	\node (Q21) [right=1 of Q12] {
  \begin{tikzpicture}[WD]
  	\node[bb={1}{2}] (a1) {};
  	\node[bb={1}{0}, right=.5 of a1_out1] (a2) {};
  	\draw (a2_in1) -- (a1_out1);
  	\draw (a1_out2) -- +(2,0);
  	\draw (a1_in1) -- +(-.5,0);
	\end{tikzpicture}
	};
	\node (Q22) [right=.8 of Q21] {
	\begin{tikzpicture}[WD]
		\draw (0,0) -- (2,0);
	\end{tikzpicture}
	};	
	\node[label=above:{\tiny unital}] at ($(Q21.east)!.5!(Q22.west)$) {$=$};
	\node (Q31) [right=1 of Q22] {
	\begin{tikzpicture}[WD]
		\node[bb={1}{2}] (a1) {};
		\node[bb={1}{2}, minimum height=1ex, right=.5 of a1_out1] (a2) {};
		\draw (a2_in1) -- (a1_out1);
		\draw (a1_out2) -- +(2,0);
		\draw (a2_out1) -- +(.5,0);
		\draw (a2_out2) -- +(.5,0);
		\draw (a1_in1) -- +(-.5,0);
	\end{tikzpicture}
	};
	\node (Q32) [right=.8 of Q31] {
	\begin{tikzpicture}[WD]
		\node[bb={1}{2}] (a1) {};
		\node[bb={1}{2}, minimum height=1ex, right=.5 of a1_out2] (a2) {};
		\draw (a2_in1) -- (a1_out2);
		\draw (a1_out1) -- +(2,0);
		\draw (a2_out1) -- +(.5,0);
		\draw (a2_out2) -- +(.5,0);
		\draw (a1_in1) -- +(-.5,0);
	\end{tikzpicture}
	};
	\node[label=above:{\tiny associative}] at ($(Q31.east)!.5!(Q32.west)$) {$=$};
\end{tikzpicture}
\]
Not only is every object $r\in\rel$ equipped with these operations $\epsilon_r,\delta_r$, but they are coherent with respect to $\rel$'s monoidal structure. By this we mean first that the operations assigned to the monoidal unit are coherence isomorphisms: $\epsilon_I=\id_I$ and $\delta_I=\rho_I=\lambda_I$, where $\rho$ and $\lambda$ are the right and left unitors. Second, for any $r,s\in\rel$, the operations interact appropriately with the monoidal product:
\[
  \epsilon_{r\otimes s}=\epsilon_r\otimes\epsilon_s
  \qqand
  \delta_{r\otimes s}=(\delta_r\otimes\delta_s)\cp(\id_r\otimes\gamma_{r,s}\otimes\id_s).
\]
One notices immediately the need for a symmetry isomorphism $\gamma_{r,s}$ in the second equation, which is there so that the codomains agree ($r\otimes s\otimes r\otimes s$). In pictures:
\begin{equation} \label{eqn.compat_tensor}
\begin{tikzpicture}[baseline=(current bounding box.center)]
	\node (P1) {
	\begin{tikzpicture}[WD]
		\node[bb={2}{0}] (rs) {$\epsilon_{r\otimes s}$};
		\draw (rs_in1) to node[above, font=\tiny] {$r$} +(-1, 0);
		\draw (rs_in2) to node[below, font=\tiny] {$s$} +(-1, 0);		
	\end{tikzpicture}
	};
	\node (P2) [right=.7 of P1] {
	\begin{tikzpicture}[WD]
		\node[bb={1}{0}] (r) {$\epsilon_r$};
		\node[bb={1}{0}, below=.5 of r] (s) {$\epsilon_s$};
		\draw (r_in1) to node[above, font=\tiny] {$r$} +(-1, 0);
		\draw (s_in1) to node[below, font=\tiny] {$s$} +(-1, 0);
	\end{tikzpicture}
	};
	\node at ($(P1.east)!.5!(P2.west)$) {$=$};
	\node (P3) [right=2 of P2] {
	\begin{tikzpicture}[WD, bb port sep=.7]
		\node[bb={2}{4}] (rs) {$\delta_{r\otimes s}$};
		\begin{scope}[font=\tiny]
  		\draw (rs_in1) to node[above] {$r$} +(-1, 0);
  		\draw (rs_in2) to node[below] {$s$} +(-1, 0);
  		\draw (rs_out1) to node[above=-1pt] {$r$} +(1, 0);
  		\draw (rs_out2) to node[above=-1pt] {$s$} +(1, 0);
  		\draw (rs_out3) to node[below=-2pt] {$r$} +(1, 0);
  		\draw (rs_out4) to node[below=-2pt] {$s$} +(1, 0);
		\end{scope}
	\end{tikzpicture}
	};
	\node (P4) [right=.7 of P3] {
	\begin{tikzpicture}[WD, bb port sep=.7]
		\node[bb={1}{2}] (r) {$\delta_r$};
		\node[bb={1}{2}, below=.5 of r] (s) {$\delta_s$};
		\coordinate (r2) at ($(s_out1)+(2,0)$);
		\coordinate (s1) at ($(r_out2)+(2,0)$);
		\begin{scope}[font=\tiny]
			\draw (r_in1) to node[above] {$r$} +(-1, 0);
			\draw (s_in1) to node[below] {$s$} +(-1, 0);
			\draw (r_out1) -- +(2.5, 0) to node[above=-1pt] {$r$} +(.5, 0);
			\draw (r_out2) -- +(1, 0) to[out=0, in=180] (r2) to node[below=-1pt] {$r$} +(.5, 0);
			\draw (s_out1) -- +(1, 0) to[out=0, in=180] (s1) to node[above=-2pt] {$s$} +(.5, 0);
			\draw (s_out2) -- +(2.5, 0) to node[below=-2pt] {$s$} +(.5, 0);
    \end{scope}		
	\end{tikzpicture}
	};
	\node at ($(P3.east)!.5!(P4.west)$) {$=$};
\end{tikzpicture}
\end{equation}
The point is that every object in $\rel$ has this commutative comonoid structure, and the operations are coherent with $I$ and $\otimes$. In this situation we will say that $\rel$ \emph{supplies commutative comonoids}.

In general, we may talk of algebraic structures on a object being defined by a prop $\pp$. A prop is a strict symmetric monoidal category whose monoid of objects is $(\nn,0,+)$; in other words, a prop is a single-sorted symmetric monoidal theory. In the case of commutative comonoids, the relevant prop is the skeleton of $\finset\op$. Indeed, $\epsilon$ and $\delta$ represent the (opposites of) the unique functions $\varnothing\to\{1\}$ and $\{1,2\}\to\{1\}$, respectively.

We say that a symmetric monoidal category $\cat{C}$ \emph{supplies} $\pp$ if every object of $\cat{C}$ is equipped with the structure of $\pp$ in a way compatible with $\cat{C}$'s monoidal structure. This notion appears frequently in recent literature. One reason for this is that the compatibility with the monoidal product is a useful and intuitive feature when adding extra icons---``bells and whistles''---into the standard monoidal category string diagram language, yielding equations such as those in \cref{eqn.compat_tensor}.

Examples abound. To list a few: categories supplying comonoids figure strongly in categorical approaches to probability theory \cite{fong2012causal,fritz2019synthetic,cho2019disintegration}; categories supplying frobenius monoids---known as hypergraph categories---are important in networks and wiring-diagram languages \cite{Carboni:1991a,fong2019hypergraph}; categories supplying bimonoids underlie a categorical perspective on differentiation \cite{blute2009cartesian}; categories supplying so-called adjoint frobenius monoids and abelian relations underlie alternative approaches to regular and abelian categories \cite{fong2019regular,fong2019abelian}, and so on. Indeed, if we add an extra condition, which we call \emph{homomorphic supply}, categories homomorphically supplying commutative comonoids are simply cartesian monoidal categories \cite{fox1976coalgebras}, and categories homomorphically supplying bimonoids are those where the product is a biproduct.

Yet despite this plethora of examples, a general definition of supply has not yet been given: to now do so is the first goal of this article. If $\pp$ is a prop and $\cat{C}$ is a symmetric monoidal category, we define what it means for $\cat{C}$ to supply the algebraic structure encoded in $\pp$. We also define what it means for a strong monoidal functor $\cat{C}\to\cat{D}$ to \emph{preserve supplies}, i.e.\ to send a given supply of $\pp$ in $\cat{C}$ to a given supply of $\pp$ in $\cat{D}$. We give a number of examples both of supply and supply preservation.

The second goal is to provide some basic theory of supply. For example, given a supply of $\pp$ in $\cat{C}$ and a prop functor $\pp'\to\pp$, one obtains a supply of $\pp'$ in $\cat{C}$. We show that if $\cat{C}$ and $\cat{D}$ both supply $\pp$ then so does their biproduct $\cat{C}\oplus\cat{D}$, and that the projections and coprojections preserve supplies. Finally, any supply of $\pp$ in $\cat{C}$ induces a supply of $\pp$ in the strictification $\strict{\cat{C}}$, and it is preserved by the equivalence $\strict{\cat{C}}\to\cat{C}$.

We also discuss what it means for various maps in $(\cat{C},I,\otimes)$ to be \emph{homomorphisms} for the supplied structure. For example, it is well-known that if $\cat{C}$ is cartesian monoidal (i.e.\ if the monoidal product is given by the categorical product) then it supplies comonoids and every morphism $f\colon c\to d$ is a comonoid homomorphism, in the sense that the following diagrams commute:
\[
\begin{tikzcd}[column sep=small]
	c\ar[rr, "f"]\ar[rd, "\epsilon_c"']&&
	d\ar[dl, "\epsilon_d"]\\&
	I
\end{tikzcd}
\hspace{.7in}
\begin{tikzcd}
	c\ar[r, "f"]\ar[d, "\delta_c"']&
	d\ar[d, "\delta_d"]\\
	c\otimes c\ar[r, "f\otimes f"']&
	d\otimes d
\end{tikzcd}
\]
Again in pictures:
\begin{equation}\label{eqn.homo}
\begin{tikzpicture}[baseline=(P1)]
	\node (P1) {
	\begin{tikzpicture}[WD]
		\node[bb={1}{0}] (e) {$\epsilon_c$};
		\draw (e_in1) -- +(-1, 0);
	\end{tikzpicture}
	};
	\node (P2) [right=.7 of P1] {
	\begin{tikzpicture}[WD]
		\node[bb={1}{1}] (f) {$f$};
		\node[bb={1}{0}, right=.5 of f] (e) {$\epsilon_d$};
		\draw (f_in1) -- +(-1, 0);
		\draw (f_out1) -- (e_in1);
	\end{tikzpicture}
	};
	\node at ($(P1.east)!.5!(P2.west)$) {$=$};
	\node (P3) [right=2 of P2] {
	\begin{tikzpicture}[WD]
		\node[bb={1}{1}] (f) {$f$};
		\node[bb={1}{2}, right=.5 of f] (d) {$\delta_d$};
		\draw (f_in1) -- +(-1, 0);
		\draw (f_out1) -- (d_in1);
		\draw (d_out1) -- +(1, 0);
		\draw (d_out2) -- +(1, 0);
  \end{tikzpicture}	
	};
	\node (P4) [right=.7 of P3] {
	\begin{tikzpicture}[WD]
		\node[bb port sep=1.2, bb={1}{2}] (d) {$\delta_c$};
		\node[bb={1}{1}, right=.5 of d_out1, font=\tiny] (f1) {$f$};
		\node[bb={1}{1}, right=.5 of d_out2, font=\tiny] (f2) {$f$};
		\draw (d_in1) -- +(-1, 0);
		\draw (d_out1) -- (f1_in1);
		\draw (d_out2) -- (f2_in1);
		\draw (f1_out1) -- +(1, 0);
		\draw (f2_out1) -- +(1, 0);
	\end{tikzpicture}
	};
	\node at ($(P3.east)!.5!(P4.west)$) {$=$};
\end{tikzpicture}
\end{equation}
These equations hold in any cartesian monoidal category, e.g.\ $\smset$, but they \emph{do not hold} in $\rel$. (As an example, the first equation in \eqref{eqn.homo} does not hold in $\rel$. Take $c=d=1$, take $f\coloneqq\varnothing\ss 1\times 1$ to be the empty relation, and note that $\epsilon_c \neq (f\cp\epsilon_d)$.) We will show that the morphisms in $\cat{C}$ that are homomorphisms for the $\pp$-structure always form a monoidal subcategory. The above notion of homomorphic supply simply refers to the case that this subcategory is all of $\cat{C}$.

The main theorems of this paper are that supply and supply preservation are well-behaved with respect to coherence isomorphisms. In \cref{thm.coherence_isos_homos} we show that every associator and unitor in $\cat{C}$ is automatically a homomorphism for any supply. In \cref{thm.pres_supp_strongators_homo} we show that the coherence isomorphisms for strong monoidal functors $F\colon\cat{C}\to\cat{D}$ are automatically homomorphisms whenever $F$ preserves the supply. We give strictification theorems \cref{thm.supply_strictification,cor.strict_equiv_pres_supply}.
 
% Subsection %
\paragraph{Acknowledgements.}
We thank Dmitry Vagner for interesting and useful conversations. We also thank Tobias Fritz for catching an error in a draft and for suggesting the strictification theorem (\cref{thm.supply_strictification}), which he had previously proven independently in a special case and shared with us; see \cite{fritz2019synthetic} forthcoming. We acknowledge support from Honeywell Inc., and from AFOSR grants FA9550-17-1-0058 and FA9550-19-1-0113.

%======== Chapter ========%
\chapter{Notation and background}

% Subsection %
\paragraph{Basic notation.}
For a natural number $n\in\nn$ we denote the corresponding ordinal by $\ord{n}=\{1,\ldots,n\}\in\smset$. %In the introduction we wrote $1$ for what we will henceforth denote $\ord{1}$.
We denote composition of $f\colon a\to b$ and $g\colon b\to c$ by $(f\cp g)\colon a\to c$, i.e.\ we use diagrammatic order. When $c$ is an object we denote the identity morphism on it either by $c$ or by $\id_c$.

% Subsection %
\paragraph{Symmetric monoidal categories and coherence.}
Suppose $(\cat{C}, I, \otimes)$ is a symmetric monoidal category, $m\in\nn$ is a natural number, and $c\colon\ord{m}\to\cat{C}$ is a family of objects in $\cat{C}$. We denote
\begin{equation}\label{eqn.big_mon_prod}
	\bigotimes_{i\in\ord{m}}c(i)\coloneqq
  \big((c(1)\otimes c(2))\cdots\big)\otimes c(m)
\end{equation}
with the convention that when $m=0$ and $!\colon\ord{0}\to \cat{C}$ is the unique function, we put $\bigotimes !\coloneqq I$. We take this to be the canonical bracketing, so $c\otimes d\otimes e$ denotes $(c\otimes d)\otimes e$. If there exists $b\in\cat{C}$ such that $b=c(i)$ for all $i\in\ord{m}$, we denote the monoidal product in \eqref{eqn.big_mon_prod} by $b\tpow{m}\coloneqq\bigotimes_{i\in\ord{m}}b$. 

If $m,n\in\nn$ are natural numbers, and $c\colon \ord{m}\times \ord{n}\to\cat{C}$ is a family of objects in $\cat{C}$, we also have a natural isomorphism
\begin{equation}\label{eqn.symmetry}
\sigma\colon
\bigotimes_{i\in\ord{m}}\bigotimes_{j\in\ord{n}}c(i,j)\Too{\cong}
\bigotimes_{j\in\ord{n}}\bigotimes_{i\in\ord{m}}c(i,j).
\end{equation}
We refer to $\sigma$ as the \emph{symmetry} isomorphism, though note that it involves associators and unitors too, not just the symmetric braiding. We will be interested in two particular cases of the symmetry isomorphism \cref{eqn.symmetry}, namely for $m=2$ and $m=0$ and any $n\in\nn$:
\[\sigma\colon c_1\tpow{n}\otimes c_2\tpow{n}\Too{\cong}(c_1\otimes c_2)\tpow{n}
\qqand
\sigma\colon I\Too{\cong} I\tpow{n}.
\]

Many of our results will rely on Mac Lane's coherence theorem for symmetric monoidal categories \cite[Theorem XI.1]{MacLane:1998a}, which says the following. For any two ways to arrange brackets and monoidal units into a word with $n$ placeholders for objects in $\cat{C}$, and for each permutation of $n$ letters, there is an associated natural isomorphism, which Mac Lane calls the \emph{canonical isomorphism}, between the resulting functors $\cat{C}^n\to\cat{C}$. Moreover, composites and tensor products of canonical isomorphisms are again canonical. For example, everything we called a symmetry isomorphisms $\sigma$ in \cref{eqn.symmetry} is one of these canonical isomorphisms.

% Subsection %
\paragraph{The 2-category $\mathbb{S}\Cat{MC}$.}
Recall that a strong monoidal functor $(F,\varphi)\colon\cat{C}\to\cat{D}$ consists of a functor $F$ and natural isomorphisms
\[
	\varphi\colon I_{\cat{D}}\To{\cong} F(I_{\cat{C}})
	\qqand
	\varphi_{c,c'}\colon F(c)\otimes_{\cat{D}} F(c')\To{\cong} F(c\otimes_{\cat{C}} c').\]
	We refer to these isomorphism as the \emph{strongators} for $F$. A strong monoidal functor is \emph{strict} if all strongators are identities.

\begin{definition}\label{def.smf}
Let $\cat{C}$ and $\cat{D}$ be symmetric monoidal categories. Define $\smf(\cat{C},\cat{D})$ to be the symmetric monoidal category whose objects are strong monoidal functors $(F,\varphi)\colon\cat{C}\to\cat{D}$, whose morphisms are monoidal natural transformations, and whose symmetric monoidal structure is given pointwise.

Define $\ssmc$ to be the 2-category whose objects are symmetric monoidal categories and whose hom-categories are given by $\smf$.
\end{definition}

The pointwise condition in \cref{def.smf} means that the monoidal unit in $\smf(\cat{C},\cat{D})$ is given by the constant functor at the monoidal unit of $\cat{D}$ and that the monoidal product is given by $(F\otimes G)(c)\coloneqq F(c)\otimes G(c).$ The strongator of $F\otimes G$ for any $c,c'\in\cat{C}$ is given by the symmetry isomorphism
\[
	\sigma\colon
	\big(F(c)\otimes G(c)\big)\otimes\big(F(c')\otimes G(c')\big)
	\To{\cong}
	\big(F(c)\otimes F(c')\big)\otimes\big(G(c')\otimes G(c')\big).
\]

\cref{thm.smc_biprod,prop.biprod_smf_strict,thm.smc_all_prod_coprod} are not necessary for the main thrust of this note, so we will save their proofs for later; see \cref{chap.proofs}. %, pages \pageref{page.smc_biprod} and \pageref{page.biprod_smf_strict}, and page \pageref{page.smc_all_prod_coprod}.
However, they seem important to us, and not sufficiently well known. 

\begin{theorem}\label{thm.smc_all_prod_coprod}
The 2-category $\ssmc$ has all small products and coproducts, and products are strict.
\end{theorem}

In fact, finite products and coproducts coincide in $\ssmc$.

\begin{theorem}\label{thm.smc_biprod}
The 2-category $\ssmc$ of symmetric monoidal categories, strong monoidal functors, and monoidal natural transformations has 2-categorical biproducts.
\end{theorem}

We denote the biproduct of symmetric monoidal categories $\cat{C}$ and $\cat{D}$ by $\cat{C}\oplus\cat{D}$.

\begin{proposition}\label{prop.biprod_smf_strict}
Let $\cat{C}_1,\cat{C}_2,\cat{D}_1,\cat{D}_2$ be symmetric monoidal categories. The functor
\[
  \oplus\colon
  \smf(\cat{C}_1,\cat{D}_1)\times\smf(\cat{C}_2,\cat{D}_2)
  \to
  \smf(\cat{C}_1\oplus\cat{C}_2,\cat{D}_1\oplus\cat{D}_2)
\]
is strict monoidal.
\end{proposition}

%======== Chapter ========%
\chapter{Supply}

In \cref{sec.supply} we define supply and give some first examples. \cref{sec.main_thm} then proves our main theorem, \cref{thm.coherence_isos_homos}: coherence isomorphisms are supply homomorphisms. We also provide a more compact definition of supply in \cref{thm.supply_v2}. In \cref{sec.further_theory}, we record some useful ways to construct new supplies from old. 

%==== Section ====%
\section{Definition of supply}\label{sec.supply}

Recall that a prop $\pp$ is a symmetric strict monoidal category whose monoid of objects is $(\nn,0,+)$. We denote its objects by $m$, $n$, etc.

\begin{definition}[Supply]\label{def.supply}
Let $\pp$ be a prop and $\cat{C}$ a symmetric monoidal category. A \emph{supply of $\pp$ in $\cat{C}$} consists of a strong monoidal functor $s_c\colon\pp\to\cat{C}$ for each object $c\in\cat{C}$, such that
\begin{enumerate}[label=(\roman*)]
	\item $s_c(m)=c\tpow{m}$ for each $m\in\nn$, 
	\item the strongator $c\tpow{m}\otimes c\tpow{n}\to c\tpow{(m+n)}$ is the unique coherence isomorphism for each $m,n\in\nn$, and
	\item the following diagrams commute for every $c,d\in\cat{C}$ and $\mu\colon m\to n$ in $\pp$:
\begin{equation}\label{eqn.supply_commute_tensors}
\begin{tikzcd}[column sep=55pt]
	c\tpow{m}\otimes d\tpow{m}\ar[r, "s_c(\mu)\otimes s_d(\mu)"]\ar[d, "\sigma"']&
	c\tpow{n}\otimes d\tpow{n}\ar[d, "\sigma"]\\
	(c\otimes d)\tpow{m}\ar[r, "s_{c\otimes d}(\mu)"']&
	(c\otimes d)\tpow{n}
\end{tikzcd}
\hspace{.7in}
\begin{tikzcd}
	I\ar[r, equal]\ar[d, "\sigma"']&
	I\ar[d, "\sigma"]\\
	I\tpow{m}\ar[r, "s_I(\mu)"']&
	I\tpow{n}
\end{tikzcd}
\end{equation}
where the $\sigma$'s are the symmetry isomorphisms from \cref{eqn.symmetry}.
\end{enumerate}
We further say that $f\colon c\to d$ is an \emph{$s$-homomorphism} if the following diagram commutes for all $\mu\colon m\to n$ in $\pp$:
\begin{equation}\label{eqn.nat_means_homo}
\begin{tikzcd}
	c\tpow{m}\ar[r, "s_c(\mu)"]\ar[d, "f\tpow{m}"']&
	c\tpow{n}\ar[d, "f\tpow{n}"]\\
	d\tpow{m}\ar[r, "s_d(\mu)"']&
	d\tpow{n}
\end{tikzcd}
\end{equation}
If every morphism in $\cat{C}$ is an $s$-homomorphism, we say that $s$ is a \emph{homomorphic supply}.
\end{definition}

\begin{remark}\label{rem.strict}
Note that if $s_c$ is strict, then conditions (i) and (ii) can be replaced by the condition $s_c(1)=c$. Moreover, if $\cat{C}$ is strict, then each $s_c$ must be too.
\end{remark}

\begin{example}\label{ex.supply_ids}
Let $\bij$ denote the initial prop: $\bij(m,n)$ is the set of bijections $\ord{m}\cong\ord{n}$. For any monoidal category $\cat{C}$ and object $c\in\cat{C}$, there is a unique strong monoidal functor $s_c\colon\bij\to\cat{C}$ satisfying conditions (i) and (ii) of \cref{def.supply}. The only morphisms in $\bij$ are the symmetries, and a computation shows that condition (iii) holds. Every morphism in $\cat{C}$ is an $s$-homomorphism. 

One might thus say that every symmetric monoidal category $\cat{C}$ uniquely \emph{supplies symmetries}, and every morphism in $\cat{C}$ is a homomorphism for symmetries.

Note that condition (ii) of \cref{def.supply} is necessary for the supply of $\bij$ to be unique.
\end{example}

\begin{remark}
One could give an alternative definition of supply by dropping conditions (i) and (ii) from \cref{def.supply}. This relaxed definition should be equivalent to ours in an appropriate sense. However we made the choice we did in order to cut down on the number of equivalent supplies. For example, we appreciate the fact from \cref{ex.supply_ids} that there is a unique supply of symmetries in any symmetric monoidal category. Props are syntactic in nature (a monoidal category equivalent to a prop is not generally a prop), and our choice was made in order to match that syntactic aesthetic.
\end{remark}

\begin{example}\label{ex.terminal_supply}
Let $\zero=\{*\}$ denote the zero object in $\ssmc$ (see \cref{thm.smc_biprod}). For any prop $\pp$ there is a unique supply of $\pp$ in $\zero$.
\end{example}

\begin{example}[Involutions]\label{ex.supply_involutions}
Consider the prop $\ccat{I}$ whose morphisms are given as follows:
\[
  \ccat{I}(m,n)=
  \begin{cases}
  	\emptyset&\tn{ if }m\neq n\\
		\{\id_m, i_m\}&\tn{ if }m=n
  \end{cases}
 \]
 with $i_m\cp i_m=\id_m$ and $i_m+i_n=i_{m+n}$. If $\cat{C}$ supplies $\ccat{I}$, we say it \emph{supplies involutions}. That means that every object $c\in\cat{C}$ is equipped with an involution $i_c\colon c\to c$, compatible with tensor products in the sense that $i_{c\otimes d}=i_c\otimes i_d$ and $i_I=\id_I$.
 
For a morphism $f\colon c\to d$ to be an involution-homomorphisms just means that $f$ commutes with the chosen involutions, i.e.\ $f\cp i_d=i_c\cp f$.
\end{example}

\begin{example}[Initial objects] \label{ex.term_init}
	Let $\pp$ be the prop generated by a unique map $\eta\colon 0 \to 1$. The monoidal unit of a symmetric monoidal category is an initial object iff the category homomorphically supplies $\pp$. Dually, the monoidal unit is an terminal object iff the category supplies $\pp\op$.
\end{example}

\begin{example}[Monoids]\label{def.prop_monoids}
The \emph{prop for commutative monoids} is given by two generators
\begin{equation}\label{eqn.gen_monoid}
\begin{tikzpicture}[WD, baseline=(eta)]
	\node[bb={0}{1}, fill=white] (eta) {$\eta$};
	\draw (eta_out1) -- +(.8,0);
	\node[bb={2}{1}, fill=white, right=4 of eta] (mu) {$\mu$};
	\draw (mu_out1) -- +(.8,0);
	\draw (mu_in1) -- +(-.8,0);
	\draw (mu_in2) -- +(-.8,0);
\end{tikzpicture}
\end{equation}
and three equations:
\begin{equation}\label{eqn.rel_monoid}
\begin{tikzpicture}[baseline=(a)]
	\node (P11) {
	\begin{tikzpicture}[WD]
		\node[bb={2}{1}, fill=white] (a) {};
		\coordinate (a1) at ($(a_in1)-(1,0)$);
		\coordinate (a2) at ($(a_in2)-(1,0)$);
		\draw (a1) to[in=180, out=0] (a_in2);
		\draw (a2) to[in=180, out=0] (a_in1);
		\draw (a_out1) -- +(.5,0);
	\end{tikzpicture}
	};
	\node (P12) [right=.8 of P11] {
	\begin{tikzpicture}[WD]
		\node[bb={2}{1}, fill=white] (a) {};
		\draw (a_in1) -- +(-.5,0);
		\draw (a_in2) -- +(-.5,0);
		\draw (a_out1) -- +(.5,0);
	\end{tikzpicture}
	};
	\node at ($(P11.east)!.5!(P12.west)$) {$=$};
	\node (P21) [right=1 of P12] {
  \begin{tikzpicture}[WD]
  	\node[bb={2}{1}, fill=white] (a1) {};
  	\node[bb={0}{1}, fill=white, left=.5 of a1_in1] (a2) {};
  	\draw (a2_out1) -- (a1_in1);
  	\draw (a1_in2) -- +(-2,0);
  	\draw (a1_out1) -- +(.5,0);
	\end{tikzpicture}
	};
	\node (P22) [right=.8 of P21] {
	\begin{tikzpicture}[WD]
		\draw (0,0) -- (2,0);
	\end{tikzpicture}
	};	
	\node at ($(P21.east)!.5!(P22.west)$) {$=$};
	\node (P31) [right=1 of P22] {
	\begin{tikzpicture}[WD]
		\node[bb={2}{1}, fill=white] (a1) {};
		\node[bb={2}{1}, fill=white, minimum height=1ex, left=.5 of a1_in1] (a2) {};
		\draw (a2_out1) -- (a1_in1);
		\draw (a1_in2) -- +(-2,0);
		\draw (a2_in1) -- +(-.5,0);
		\draw (a2_in2) -- +(-.5,0);
		\draw (a1_out1) -- +(.5,0);
	\end{tikzpicture}
	};
	\node (P32) [right=.8 of P31] {
	\begin{tikzpicture}[WD]
		\node[bb={2}{1}, fill=white] (a1) {};
		\node[bb={2}{1}, fill=white, minimum height=1ex, left=.5 of a1_in2] (a2) {};
		\draw (a2_out1) -- (a1_in2);
		\draw (a1_in1) -- +(-2,0);
		\draw (a2_in1) -- +(-.5,0);
		\draw (a2_in2) -- +(-.5,0);
		\draw (a1_out1) -- +(.5,0);
	\end{tikzpicture}
	};
	\node at ($(P31.east)!.5!(P32.west)$) {$=$};
\end{tikzpicture}
\end{equation}
It is equivalent to the skeleton of $\finset$, i.e. with $\Hom(m,n)\coloneqq\smset(\ord{m},\ord{n})$. For example, the generators shown in \cref{eqn.gen_monoid} correspond to the unique functions $\ord{0}\to\ord{1}$ and $\ord{2}\to\ord{1}$ respectively.

A supply of commutative monoids in $\cat{C}$ gives a map $\mu_c\colon c\otimes c\to c$ and $\eta_c\colon I\to c$ for each object $c$, compatible with tensor product in $\cat{C}$ and satisfying the usual monoid laws. A morphism $f\colon c\to d$ is a monoid-homomorphism in the sense of \cref{def.supply} iff it is in the usual sense: $\mu_c\cp f=(f\otimes f)\cp \mu_d$ and $\eta_c\cp f=\eta_d$.

Similarly, to supply commutative comonoids means to supply the prop given by the skeleton of $\finset\op$.
\end{example}

\begin{example}[Cartesian categories]\label{ex.cart_grant_comonoids}
A symmetric monoidal category has finite products iff it homomorphically supplies commutative comonoids. In this case, the categorical product coincides with the monoidal product. This was shown in \cite{fox1976coalgebras}.
\end{example}

\begin{example}[Compact closed categories]\label{ex.self_duals}
The \emph{prop $\dd$ for self-duals} has two generators
\[
\begin{tikzpicture}[WD, font=\tiny, light gray nodes]
	\node[bb={2}{0}] (a) {};
	\node[bb={0}{2}, right=5 of a] (b) {};
	\draw (a_in1) -- +(-.5, 0);
	\draw (a_in2) -- +(-.5, 0);
	\draw (b_out1) -- +(.5, 0);
	\draw (b_out2) -- +(.5, 0);
	\node[fill=white, inner sep=1pt, right=10pt of b] {};
	\node[fill=white, font=\normalsize] at ($(a.east)!.5!(b.west)$) {and};
\end{tikzpicture}
\]
and four equations
\begin{equation}\label{eqn.self_duals}
\begin{tikzpicture}
	\node (Q1) {
\begin{tikzpicture}[WD, font=\tiny, bb port length=.5, light gray nodes]
	\node[bb={2}{0}] (a) {};
\end{tikzpicture}	
	};
	\node (Q2) [right=.5 of Q1] {
\begin{tikzpicture}[WD, font=\tiny, bb port length=.2, light gray nodes]
	\node[bb={2}{0}] (a) {};
	\coordinate (end) at ($(a_in1)+(-1,0)$);
	\draw (a_in1) to[out=180, in=0] (a_in2-|end);
	\draw (a_in2) to[out=180, in=0] (a_in1-|end);
\end{tikzpicture}	
	};
	\node (Q3) [right=.9 of Q2]{
\begin{tikzpicture}[WD, font=\tiny, bb port length=.5, light gray nodes]
	\node[bb={0}{2}] (a) {};
\end{tikzpicture}	
	};
	\node (Q4) [right=.5 of Q3] {
\begin{tikzpicture}[WD, font=\tiny, bb port length=.2, light gray nodes]
	\node[bb={0}{2}] (a) {};
	\coordinate (end) at ($(a_out1)+(1,0)$);
	\draw (a_out1) to[out=0, in=180] (a_out2-|end);
	\draw (a_out2) to[out=0, in=180] (a_out1-|end);
\end{tikzpicture}	
	};
  \node at ($(Q1.east)!.5!(Q2.west)$) {$=$};
  \node at ($(Q3.east)!.5!(Q4.west)$) {$=$};
	\node (P1) [right=.9 of Q4] {
  \begin{tikzpicture}[WD, light gray nodes]
  	\node[bb={0}{2}] (a) {};
  	\node[bb={2}{0}, above right=-.75 and .5 of a] (b) {};
  	\draw (a_out1) -- (b_in2);
  	\draw (a_out2) -- +(2, 0);
  	\draw (b_in1) -- +(-2, 0);
  \end{tikzpicture}
  };
  \node (P2) [right=.5 of P1] {
  \begin{tikzpicture}[WD]
  	\draw (0,0) -- (1.5,0);
  \end{tikzpicture}
  };
  \node at ($(P1.east)!.5!(P2.west)$) {$=$};
	\node (P3) [right=.9 of P2]{
  \begin{tikzpicture}[WD, light gray nodes]
  	\node[bb={0}{2}] (a) {};
  	\node[bb={2}{0}, below right=-.75 and .5 of a] (b) {};
  	\draw (a_out2) -- (b_in1);
  	\draw (a_out1) -- +(2, 0);
  	\draw (b_in2) -- +(-2, 0);
  \end{tikzpicture}
  };
  \node (P4) [right=.5 of P3] {
  \begin{tikzpicture}[WD]
  	\draw (0,0) -- (1.5,0);
  \end{tikzpicture}
  };
  \node at ($(P3.east)!.5!(P4.west)$) {$=$};
\end{tikzpicture}
\end{equation}
$\dd$ is equivalent to the symmetric monoidal category of unoriented 1-cobordisms, though we will not need that fact.

A category suppling self-duals is called a \emph{self-dual compact closed} category.
\end{example}

\begin{example}\label{ex.mat}
  Let $\Cat{Mat}$ be the symmetric monoidal category with finite-dimensional real vector spaces $\mathbb{R}^n$ as objects, matrices between them as morphisms, Kronecker product $\otimes$ of matrices as tensor product. Then $\Cat{Mat}$ supplies self-duals, and the supply homomorphisms are orthogonal matrices.
\end{example}

\begin{example}\label{ex.frob_mon}
The prop for (special, commutative) frobenius monoids is $\cospan$, the category of cospans in $\finset$. A category supplying frobenius monoids is called a \emph{hypergraph} category; see \cite{fong2019hypergraph}.
\end{example}

\begin{proposition}\label{prop.p_supplies_itself}
Let $\pp$ be a prop. Then there is a supply of $\pp$ in $\pp$.
\end{proposition}
\begin{proof}
The monoidal product in a prop is denoted $+$; we denote the $n$-fold monoidal product of some $k$ by $k\cdot n\coloneqq k+\Cdots{n}+k$.

For any $k\in\pp$ let $s_k\colon\pp\to\pp$ act on objects by $s_k(n)=k\cdot n$; this is strict because $s_k(m+n)=k\cdot(m+n)=(k\cdot m)+(k\cdot n)$. Given $\mu\colon m\to n$ in $\pp$, define $s_k(\mu)\colon k\cdot m\to k\cdot n$ by conjugating with the symmetries and applying $\mu$, on each of the $k$ factors:
\begin{equation}\label{eqn.conjugation}
	k\cdot m\To{\sigma_{k,m}}
	m\cdot k\To{\mu\cdot k}
	n\cdot k\To{\sigma_{n,k}}
	k\cdot n.
\end{equation}
This is functorial because $\sigma_{n,k}\cp\sigma_{k,n}=\id_{n\cdot k}$. It is an easy exercise to show that the diagrams in \cref{eqn.supply_commute_tensors} commute for any $k,\ell\in\pp$.
\end{proof}

Recall that the 2-category $\ssmc$ has coproducts (\cref{thm.smc_all_prod_coprod}). It is sometimes useful to note the following basic fact, which follows immediately from the definition of supply (\cref{def.supply}) and the universal property of coproducts.

\begin{proposition}
A supply $s$ of $\pp$ in $\cat{C}$ induces a strong monoidal functor $\bigsqcup_{c\in\ob(\cat{C})}\pp\to\cat{C}$ that is surjective on objects.
\end{proposition}

%==== Section ====%
\section{An equivalent definition}\label{sec.main_thm}

In this section we prove our main theorem, \cref{thm.coherence_isos_homos}, which says that all coherence isomorphisms---associators, unitors, and braiding---are homomorphisms for any supply. We use it to provide a slightly more compact definition of supply in \cref{thm.supply_v2}.

\begin{theorem}\label{thm.coherence_isos_homos}
Suppose $s$ is a supply of $\pp$ in $\cat{C}$. All the coherence isomorphisms in $\cat{C}$ (associators, unitors, and braiding) are $s$-homomorphisms.
\end{theorem}
\begin{proof}
Choose any $\mu\colon m\to n$ in $\pp$. We need to show that whenever $f\colon x\to y$ is an associator, a unitor, or a braiding, the following diagram commutes:
\[
 \begin{tikzcd}
	x\tpow{m}\ar[d, "s_x(\mu)"']\ar[r, "f\tpow{m}"]&
	y\tpow{m}\ar[d, "s_y(\mu)"]\\
	x\tpow{n}\ar[r, "f\tpow{n}"']&
	y\tpow{n}
\end{tikzcd}
\]
When $f$ is the associator $(a\otimes b)\otimes c\to a\otimes (b\otimes c)$ we consider the following diagram:
\[
\begin{tikzcd}[row sep=24pt, font=\small, column sep=25pt]
  \big((a\otimes b)\otimes c\big)\tpow{m}\ar[d, "s(\mu)_{(a\otimes b)\otimes c}" description]&
  (a\tpow{m}\otimes b\tpow{m})\otimes c\tpow{m}\ar[l, "\sigma"']\ar[r, "\alpha"]\ar[d, "(s(\mu)_a\otimes s(\mu)_b)\otimes s(\mu)_c" description]&
  a\tpow{m}\otimes (b\tpow{m}\otimes c\tpow{m})\ar[r, "\sigma"]\ar[d, "s(\mu)_a\otimes (s(\mu)_b\otimes s(\mu)_c)" description]&
  \big(a\otimes (b\otimes c)\big)\tpow{m}\ar[d, "s(\mu)_{a\otimes (b\otimes c)}" description]
  \\
  \big((a\otimes b)\otimes c\big)\tpow{n}&
  (a\tpow{n}\otimes b\tpow{n})\otimes c\tpow{n}\ar[l, "\sigma"]\ar[r, "\alpha"']&
  a\tpow{n}\otimes (b\tpow{n}\otimes c\tpow{n})\ar[r, "\sigma"']&
  \big(a\otimes (b\otimes c)\big)\tpow{n}
\end{tikzcd}
\]
The left- and right-hand squares commute by two applications of the left-hand diagram in \cref{eqn.supply_commute_tensors}, while the center square is just the naturality of the associator. Replacing the leftward horizontal maps by their inverses, Mac Lane's coherence theorem implies the composite horizontal maps are simply the relevant tensor powers of associators. Moreover, the diagram still commutes, and hence associators are $s$-homomorphisms.

The argument that braidings are homomorphisms is strictly analogous to the above. The argument that unitors are homomorphisms is almost analogous, but the proof requires also the commutativity of the right-hand diagram in \cref{eqn.supply_commute_tensors}. Indeed, consider the following diagram:
\[
\begin{tikzcd}[row sep=30pt]
	(a\otimes I)\tpow{m}\ar[d, "s(\mu)_{a\otimes I}"']&
	a\tpow{m}\otimes I\tpow{m}\ar[d, "s(\mu)_a\otimes s(\mu)_I" description]\ar[l, "\sigma"']&[10pt]
	a\tpow{m}\otimes I\ar[d, "s(\mu)_a\otimes I" description]\ar[r, "\rho"]\ar[l, "a\tpow{m}\otimes\sigma"']&
	a\tpow{m}\ar[d, "s(\mu)_a"]\\
	(a\otimes I)\tpow{n}&
	a\tpow{n}\otimes I\tpow{n}\ar[l, "\sigma"]&
	a\tpow{n}\otimes I\ar[l, "a\tpow{n}\otimes\sigma"]\ar[r, "\rho"']&
	a\tpow{n}
\end{tikzcd}
\]
Its left-hand and middle diagrams commute by \cref{eqn.supply_commute_tensors} and the right-hand diagram commutes by the unitor axiom.
\end{proof}

We can use \cref{thm.coherence_isos_homos} to provide a more compact definition of supply. To do so, we need the following definition, which puts all the coherence isomorphisms in $\cat{C}$ into a single monoidal subcategory, denoted $\mob{\cat{C}}$.

\begin{definition}\label{def.mob}
For any symmetric monoidal category $\cat{C}$, let $\mob{\cat{C}}\ss\cat{C}$ denote the smallest subcategory containing
\begin{enumerate}[label=(\roman*)]
	\item all objects of $\cat{C}$ (and identity morphisms), and
	\item all coherence isomorphisms---unitors, associators, braiding, and their inverses---from $\cat{C}$.
\end{enumerate}
Thus $\mob{\cat{C}}$ inherits a symmetric monoidal structure, and we refer to it as the \emph{symmetric monoidal category of $\cat{C}$-objects}. There is an identity-on-objects inclusion $\inc\colon\mob{\cat{C}}\to\cat{C}$. 
\end{definition}

\begin{example}
When $\cat{C}$ is strict monoidal, $\mob{\cat{C}}=\ob(\cat{C})$ is discrete.
\end{example}

The reader may find it useful to consider the meaning of $\inc\tpow{m}\in\smf(\mob{\cat{C}},\cat{C})$ for $m\in\nn$. In particular it sends $c\mapsto c\tpow{m}=((c\otimes c)\otimes\cdots\otimes c)\otimes c$ and its strongators are the symmetry isomorphisms; see \cref{def.smf}.

\begin{theorem}\label{thm.supply_v2}
There is a one-to-one correspondence between supplies $s$ of $\pp$ in $\cat{C}$ and strong monoidal functors $\tilde{s}\colon\pp\to\smf(\mob{\cat{C}},\cat{C})$ such that 
\begin{enumerate}[label=(\roman*)]
	\item $m\mapsto\inc\tpow{m}$ for each $m\in\nn$, and
	\item the strongator $\tilde{s}(m)\otimes \tilde{s}(n)\to \tilde{s}(m+n)$ is the unique coherence map for each $m,n \in \nn$.
\end{enumerate}
\end{theorem}\begin{proof}
A strong monoidal functor $\tilde{s}$ obeying (i) and (ii) is simply a supply $s$ of $\pp$ in $\cat{C}$ such that the coherence maps are $s$-homomorphisms. But \cref{thm.coherence_isos_homos} shows that every supply has this property, and so the two notions coincide. We explain this in detail.

Let $\tilde{s}$ be a strong monoidal functor obeying (i) and (ii). Note that (i) defines $\tilde{s}$ on objects. On morphisms, each $\mu\colon m \to n$ defines a monoidal natural transformation $\tilde{s}(\mu)\colon \inc\tpow{m} \Rightarrow \inc\tpow{n}$. Explicitly, this is, for each object $c \in \cat{C}$, a morphism $\tilde{s}(\mu)_c \colon c\tpow{m} \to c \tpow{n}$ obeying naturality and monoidality conditions. Naturality requires 
\begin{equation}\label{eqn.C_0_homos}
\begin{tikzcd}
	c\tpow{m}\ar[d, "s_c(\mu)"']\ar[r, "f\tpow{m}"]&
	d\tpow{m}\ar[d, "s_d(\mu)"]\\
	c\tpow{n}\ar[r, "f\tpow{n}"']&
	d\tpow{n}
\end{tikzcd}
\end{equation}
to commute for all maps $f\colon c \to d$ in $\mob{\cat{C}}$---that is, all coherence maps of $\cat{C}$---while monoidality requires the diagrams
\begin{equation}\label{eqn.supply_commute_tensors_v2}
\begin{tikzcd}[column sep=55pt]
	c\tpow{m}\otimes d\tpow{m}\ar[r, "\tilde{s}(\mu)_c\otimes \tilde{s}(\mu)_d"]\ar[d, "\sigma"']&
	c\tpow{n}\otimes d\tpow{n}\ar[d, "\sigma"]\\
	(c\otimes d)\tpow{m}\ar[r, "\tilde{s}(\mu)_{c\otimes d}"']&
	(c\otimes d)\tpow{n}
\end{tikzcd}
\hspace{.7in}
\begin{tikzcd}
	I\ar[r, equal]\ar[d, "\sigma"']&
	I\ar[d, "\sigma"]\\
	I\tpow{m}\ar[r, "\tilde{s}(\mu)_I"']&
	I\tpow{n}
\end{tikzcd}
\end{equation}
commute for all $m,n \in \nn$.

The functoriality of $\tilde{s}$ requires that for all $\mu\colon m \to n$ and $\nu\colon n\to p$ we have
\[
  \tilde{s}(\mu\cp\nu)_c=\tilde{s}(\mu)_c\cp \tilde{s}(\nu)_c
\]
while the monoidality of $\tilde{s}$ with respect to the strongators given in condition (ii) imply that for all $\mu\colon m \to n $ and $\mu'\colon m' \to n'$ we have
\[
\begin{tikzcd}[column sep=55pt]
	c\tpow{m}\otimes c\tpow{m'}\ar[r, "\tilde{s}(\mu)_c\otimes \tilde{s}(\mu')_c"]\ar[d, "\alpha"']&
	c\tpow{n}\otimes c\tpow{n'}\ar[d, "\alpha"]\\
	c\tpow{m+m'}\ar[r, "\tilde{s}(\mu+\mu')_{c}"']&
	c\tpow{n+n'}.
\end{tikzcd}  
\]

It is now straightforward to see [1] that the functoriality and monoidality of $\tilde{s}$ with respect to the strongators of condition (ii) states exactly that for each $c \in \cat{C}$ the component $\tilde{s}(-)_c$ defines a strong monoidal functor $\pp \to \cat{C}$ obeying conditions (i) and (ii) of \cref{def.supply}, [2] that the monoidality diagrams \cref{eqn.supply_commute_tensors_v2} of each natural transformation $\tilde{s}(\mu)$ are exactly the diagrams \cref{eqn.supply_commute_tensors} of condition (iii) in \cref{def.supply}, and [3] that the naturality of each $\tilde{s}(\mu)$ with respect to $\mob{\cat{C}}$ is exactly the homomorphism property of \cref{thm.coherence_isos_homos}. This proves the theorem.
\end{proof}

\begin{corollary}\label{cor.strict_supply_v2}
Let $\cat{C}$ be a symmetric strict monoidal category. There is a one-to-one correspondence between supplies $s$ of $\pp$ in $\cat{C}$ and strict monoidal functors $\tilde{s}\colon\pp\to\smf(\mob{\cat{C}},\cat{C})$ such that $1 \mapsto \inc$.
\end{corollary}

Because of the one-to-one correspondence \cref{thm.supply_v2}, we often elide the difference between the supply $s$ and the strong monoidal functor $\tilde{s}$. 

\begin{theorem}\label{thm.homos_form_subcat}
Let $s$ be a supply of $\pp$ in $\cat{C}$. Then the collection of $s$-homomorphisms forms a monoidal subcategory $\mob{\cat{C}}\ss\cat{C}_s\ss\cat{C}$, and the functor $s\colon\pp\to\smf(\mob{\cat{C}},\cat{C})$ factors through a strong monoidal functor
\[s\colon\pp\to\smf(\cat{C}_s,\cat{C})\]
satisfying the two conditions of \cref{thm.supply_v2}.
\end{theorem}
\begin{proof}
We showed in \cref{thm.coherence_isos_homos} that every coherence isomorphism in $\cat{C}$ is an $s$-homomorphism. It is obvious that if $f\colon c\to d$ and $g\colon d\to e$ are $s$-homomorphisms then so is $f\cp g$. Moreover, if $f_1\colon c_1\to d_1$ and $f_2\colon c_2\to d_2$ are $s$-homomorphisms then so is $(f_1\otimes f_2)$; this follows from \cref{eqn.supply_commute_tensors_v2}. Thus $\cat{C}_s$ forms a monoidal subcategory of $\cat{C}$, and $\mob{\cat{C}}\ss\cat{C}_s$. The factoring of $s$ through $\smf(\cat{C}_s,\cat{C})$ is just a repackaging of the statement that every morphism in $\cat{C}_s$ is an $s$-homomorphism. 
\end{proof}

%==== Section ====%
\section{Transfer of supply}\label{sec.further_theory}

In this section we present a number of propositions that describe how new supplies may be constructed from old: a supply of $\qq$ in $\cat{C}$ induces a supply of $\pp$ in $\cat{C}$ for any prop functor $\pp \to \qq$; if $\cat{C}$ and $\cat{D}$ supply $\pp$ then so does their biproduct $\cat{C}\oplus\cat{D}$; a supply transfers along an essentially surjective, strict monoidal functor $\cat{C} \to \cat{D}$; and a supply on $\cat{C}$ induces a supply on its strictification $\strict{C}$.

\begin{proposition}\label{cor.change_of_supply}
Let $F\colon\pp\to\qq$ be a prop functor. For any supply $s$ of $\qq$ in $\cat{C}$, we have a supply $(F\cp s)$ of $\pp$ in $\cat{C}$.
\end{proposition}
\begin{proof}
Given a strong monoidal functor $s\colon\qq\to\smf(\mob{\cat{C}},\cat{C})$, we compose it with $F$ (which is strict and sends $1\mapsto 1$) to get the required supply of $\pp$; see \cref{thm.supply_v2}.
\end{proof}

\begin{example}
The prop $\dd$ for self-duals was given in \cref{ex.self_duals} and that for frobenius monoids was given in \cref{ex.frob_mon}; it is $\cospan$. There is a prop functor $\dd\to\cospan$ sending the generators\;
$
\begin{tikzpicture}[WD, font=\tiny, light gray nodes, bb port sep=.3]
	\node[bb={2}{0}] (a) {};
	\draw (a_in1) -- +(-.5, 0);
	\draw (a_in2) -- +(-.5, 0);
\end{tikzpicture}
$
\;and\;
$
\begin{tikzpicture}[WD, font=\tiny, light gray nodes, bb port sep=.3]
	\node[bb={0}{2}, right=2 of a] (b) {};
	\draw (b_out1) -- +(.5, 0);
	\draw (b_out2) -- +(.5, 0);
\end{tikzpicture}
$\;
to the cospans $2\to 1\from 0$ and $0\to 1\from 2$. It is easy to check that the equations \cref{eqn.self_duals} hold in $\cospan$, i.e.\ the composites
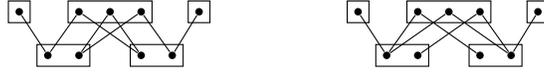
\[
\begin{tikzpicture}[x=.3cm, y=.4cm]
  \begin{scope}
  	[every node/.style={circle, fill=black, inner sep=1pt}]
  	\node (a) {};
  	\node[below right=1.25 and 1 of a] (b1) {};
  	\node[right=1 of b1] (b2) {};
  	\node at (a-|b2) (c1) {};
  	\node[right=1 of c1] (c2) {};
  	\node[right=1 of c2] (c3) {};
  	\node at (c3|-b1) (d1) {};
  	\node[right=1 of d1] (d2) {};	
  	\node[right=1] at (d2|-a) (e) {};
  \end{scope}
	\node[inner sep=2.5pt, draw, fit=(a)] {};
	\node[inner sep=2.5pt, draw, fit=(b1) (b2)] {};
	\node[inner sep=2.5pt, draw, fit=(c1) (c3)] {};
	\node[inner sep=2.5pt, draw, fit=(d1) (d2)] {};
	\node[inner sep=2.5pt, draw, fit=(e)] {};
	\draw (a) -- (b1);
	\draw (c1) -- (b1);
	\draw (c2) -- (b2);
	\draw (c3) -- (b2);
	\draw (c1) -- (d1);
	\draw (c2) -- (d1);
	\draw (c3) -- (d2);
	\draw (e) -- (d2);
  \begin{scope}
  	[every node/.style={circle, fill=black, inner sep=1pt}]
  	\node[right=2cm of e] (a) {};
  	\node[below right=1.25 and 1 of a] (b1) {};
  	\node[right=1 of b1] (b2) {};
  	\node at (a-|b2) (c1) {};
  	\node[right=1 of c1] (c2) {};
  	\node[right=1 of c2] (c3) {};
  	\node at (c3|-b1) (d1) {};
  	\node[right=1 of d1] (d2) {};	
  	\node[right=1] at (d2|-a) (e) {};
  \end{scope}
	\node[inner sep=2.5pt, draw, fit=(a)] {};
	\node[inner sep=2.5pt, draw, fit=(b1) (b2)] {};
	\node[inner sep=2.5pt, draw, fit=(c1) (c3)] {};
	\node[inner sep=2.5pt, draw, fit=(d1) (d2)] {};
	\node[inner sep=2.5pt, draw, fit=(e)] {};
	\draw (a) -- (b1);
	\draw (c1) -- (b1);
	\draw (c2) -- (b1);
	\draw (c3) -- (b2);
	\draw (c1) -- (d1);
	\draw (c2) -- (d2);
	\draw (c3) -- (d2);
	\draw (e) -- (d2);
\end{tikzpicture}
\]
are both the identity cospan $1=1=1$. Thus by \cref{cor.change_of_supply}, every hypergraph category is a self-dual compact closed category.
\end{example}

Recall from \cref{thm.smc_biprod} that the 2-category of symmetric monoidal categories has biproducts.

\begin{proposition}
If $\cat{C}$ and $\cat{D}$ both supply $\pp$ then so does their biproduct $\cat{C}\oplus\cat{D}$.
\end{proposition}
\begin{proof}
Noting that $\mob{(\cat{C}\oplus \cat{D})} =\mob{\cat{C}} \oplus \mob{\cat{D}}$, this follows from \cref{prop.biprod_smf_strict,thm.supply_v2}.
\end{proof}

We next prove that supplies transfer along strict monoidal, essentially surjective functors. Note that this assumes the axiom of choice, i.e.\ that fully faithful essentially surjective functors have inverses.

\begin{proposition}\label{cor.strong_bo}
Suppose $F\colon\cat{C}\to\cat{D}$ is a strict symmetric monoidal, essentially surjective functor. If $\cat{C}$ supplies $\pp$ then so does $\cat{D}$.
\end{proposition}
\begin{proof}
Any strict symmetric monoidal functor $F\colon\cat{C}\to\cat{D}$ induces a strict symmetric monoidal functor $\mob{F}\colon\mob{\cat{C}}\to\mob{\cat{D}}$, in fact a fully faithful functor, commuting with the inclusions. (Note that this requires \emph{strictness}; it does not in general hold for strong monoidal functors.) 

Since in this case $F$ is essentially surjective, the symmetric monoidal functor $\mob{F}$ is an equivalence, thus so is $\smf(\mob{F},\cat{D})\colon\smf(\mob{\cat{D}},\cat{D})\to\smf(\mob{\cat{C}},\cat{D})$. Now given a supply $s$ as in the following diagram (see \cref{thm.supply_v2}), one simply defines $t$ using the inverse of the equivalence $\smf(\mob{F},\cat{D})$:
\[
\begin{tikzcd}[column sep=50pt,baseline=(W.base)]
	\pp\ar[r, "s"]\ar[d, dashed, "t"']&
	\smf(\mob{\cat{C}},\cat{C})\ar[d, "{\smf(\mob{\cat{C}},F)}"]\\
	\smf(\mob{\cat{D}},\cat{D})\ar[r, "\cong"']& |[alias=W]|
	\smf(\mob{\cat{C}},\cat{D})
\end{tikzcd}
\qedhere
\]
\end{proof}

\begin{example}
Using \cref{cor.strong_bo}, a supply of comonoids on $\rel$ can be obtained from the one on $\smset$ via the bijective-on-objects (hence strict) monoidal inclusion $\smset\to\rel$. The supply homomorphisms are precisely the functional relations \cite{fong2019regular}.
\end{example}

\begin{remark}
  Supplies do not in general transfer along equivalences of categories. For an example, see \cite[Example 2.20]{fong2019hypergraph}, which gives an equivalence of symmetric monoidal categories together with a hypergraph structure on one that cannot be transferred to the other.
\end{remark}

The failure of supplies to in general transfer along equivalences notwithstanding, we close this section by proving that if $\cat{C}$ supplies $\pp$, then so does its Mac Lane strictification.

\begin{lemma}\label{lemma.strictification}
Let $\cat{C}$ a symmetric monoidal category, $\strict{\cat{C}}$ its strictification, and $\bigotimes\colon\strict{\cat{C}}\to\cat{C}$ the strong monoidal equivalence. For any prop $\pp$, a strong monoidal functor $F\colon\pp\to\cat{C}$ factors as $\pp\To{\strict{F}}\strict{\cat{C}}\To{\bigotimes}\cat{C}$, for some strict monoidal functor $\strict{F}$ iff
\begin{enumerate}[label=(\roman*)]
	\item $F(m)=F(1)\tpow{m}$ for all $m\in\nn$, and
	\item the strongator $F(m)\otimes F(n)\to F(m+n)$ is the unique coherence map for all $m,n \in \nn$.
\end{enumerate}
\end{lemma}
\begin{proof}
Clearly if $F$ factors as $\strict{F}\cp\bigotimes$ then it satisfies the two conditions. Conversely, if $F$ satisfies the two conditions, define $\strict{F}$ on objects by $\strict{F}(m)\coloneqq[F(1),\Ldots{m},F(1)]$; note that $\strict{F}(m)\cp\bigotimes=F(m)$. On morphisms define $\strict{F}$ to be the composite
\[\pp(m,n)\To{F}\cat{C}(F(m),F(n))\To{\cong}\strict{\cat{C}}(\strict{F}(m),\strict{F}(n)),\]
where the second map is the isomorphism coming from the fact that $\bigotimes$ is fully faithful. It is clear both that $\strict{F}$ is strict and that its composite with $\bigotimes$ is $F$.
\end{proof}

\begin{proposition}\label{thm.supply_strictification}
For any supply on $\cat{C}$, there is an induced supply on its strictification $\strict{\cat{C}}$.
\end{proposition}
\begin{proof}
Let $\pp$ be a prop, and suppose $s$ is a supply of $\pp$ in $\cat{C}$. For each $c\in\cat{C}$, \cref{lemma.strictification} says that the map $s_c\colon\pp\to\cat{C}$ factors through a strict monoidal functor $s_{[c]}\colon\pp\to\strict{\cat{C}}$ sending $1\mapsto[c]$. It immediately satisfies conditions (i) and (ii) of \cref{def.supply}, and it satisfies condition (iii) because $\bigotimes\colon\strict{\cat{C}}\to\cat{C}$ is faithful.

For an arbitrary object $c=[c_1,\ldots,c_k]\in\strict{C}$, define $\strict{s}_c\colon\pp\to\strict{\cat{C}}$ on each object $m\in\nn$ by $\strict{s}_c(m)\coloneqq[c,\Ldots{m},c]$ and on each morphism $\mu$ by conjugating with the symmetries:
\[
\begin{tikzcd}[column sep=90pt]
	{[c_1,\Ldots{m},c_1] \cdot \ldots \cdot [c_k,\Ldots{m}, c_k]}
		\ar[r, "{s_{[c_1]}(\mu) \cdot \ldots \cdot s_{[c_k]}(\mu)}"]&
	{[c_1,\Ldots{n},c_1] \cdot \ldots\cdot [c_k,\Ldots{n}, c_k]}
		\ar[d, "\sigma"]\\
	{[c_1,\ldots, c_k] \cdot \Ldots{m} \cdot [c_1,\ldots,c_k]}
		\ar[u, "\sigma"]\ar[r, dashed, "\strict{s}_c(\mu)"']&
	{[c_1,\ldots, c_k] \cdot \Ldots{n} \cdot [c_1,\ldots,c_k]}
\end{tikzcd}
\]
where we have written $\cdot$ for the monoidal product in $\strict{\cat{C}}$, namely list concatenation.
With this assignment, $\strict{s}$ is easily seen to be a supply of $\pp$ in $\strict{\cat{C}}$.
\end{proof}
Once we have defined preservation of supply in \cref{sec.pres_supply}, we will see immediately that the equivalence $\strict{\cat{C}}\to\cat{C}$ is supply-preserving; see \cref{cor.strict_equiv_pres_supply}.

%======== Chapter ========%
\chapter{Preservation of supply}\label{sec.pres_supply}

In \cref{sec.pres_supply} we define preservation of supply---i.e.\ the notion of \emph{homomorphism} between categories equipped with supply---and give some basic examples. In \cref{sec.preserve_thms} we prove some useful properties of supply-preserving functors. Of these, the most important is \cref{thm.pres_supp_strongators_homo}, which says that for any strong monoidal functor preserving supply, the strongators are homomorphisms.

%==== Section ====%
\section{Definition and examples}\label{sec.pres_supply}

\begin{definition}[Preserves supply]\label{def.preserve_supply}
Let $\pp$ be a prop, $\cat{C}$ and $\cat{D}$ symmetric monoidal categories, and suppose $s$ is a supply of $\pp$ in $\cat{C}$ and $t$ is a supply of $\pp$ in $\cat{D}$. We say that a strong symmetric monoidal functor $(F,\varphi)\colon\cat{C}\to\cat{D}$ \emph{preserves the supply} if the strongators $\varphi$ provide a natural isomorphism $t_{Fc}\cong (s_c\cp F)$ of functors $\pp\to\cat{D}$ for each $c\in\cat{C}$.
\end{definition}

Unpacking, a strong monoidal functor $(F,\varphi)$ preserves the supply iff the diagram
\begin{equation}\label{eqn.unpack_preserve_supply}
	\begin{tikzcd}[column sep=large]
  	F(c)\tpow{m}\ar[r, "t_{F(c)}(\mu)"]\ar[d, "\varphi"', "\cong"]&
  	F(c)\tpow{n}\ar[d, "\varphi", "\cong"']\\
  	F(c\tpow{m})\ar[r, "F(s_c(\mu))"']&
  	F(c\tpow{n})
  \end{tikzcd}
\end{equation}
commutes for each morphism $\mu\colon m\to n$ in $\pp$ and object $c\in\cat{C}$.

\begin{example}
Taking $\pp=\bij$ as in \cref{ex.supply_ids} we see that every strong monoidal functor $\cat{C}\to\cat{D}$ preserves the supply of symmetries.
\end{example}

\begin{example}
Let $s$ be a supply of $\pp$ in $\cat{C}$. Recall that there is a unique supply of $\pp$ on $\zero$ by \cref{ex.terminal_supply}. It follows from the second diagram in \cref{eqn.supply_commute_tensors} that the unique monoidal functor $\zero\to\cat{C}$ preserves the $\pp$-supply (and clearly so does $\cat{C}\to\zero$).
\end{example}

\begin{example}\label{ex.preserve_involutions}
Suppose we have a supply $s$ of involutions in $\cat{C}$ and a supply $t$ of involutions in $\cat{D}$. As we saw in \cref{ex.supply_involutions} this just means that every object $x$ is equipped with an involution $i_x\colon x\cong x$. A symmetric monoidal functor $F\colon\cat{C}\to\cat{D}$ preserves the supply iff $F(i_x)=i_{F(x)}$.
\end{example}

\begin{example}
  A hypergraph functor is defined to be a strong symmetric monoidal functor between hypergraph categories that preserves the supply of frobenius monoids.
\end{example}

\begin{remark}
  Note that \cref{def.preserve_supply} permits a straightforward generalization to lax monoidal functors. We use this stronger definition as all the examples we are aware of use strong monoidal functors, and because this structure is used in the results below.
\end{remark}

%==== Section ====%
\section{Basic theory of preservation}\label{sec.preserve_thms}

\begin{theorem}\label{thm.pres_supp_strongators_homo}
Let $s$ be a supply of $\pp$ in $\cat{C}$ and let $t$ be a supply of $\pp$ in $\cat{D}$, and suppose that $(F,\varphi)\colon\cat{C}\to\cat{D}$ is a strong monoidal functor preserving supply. Then the strongators $\varphi$ are $t$-homomorphisms, i.e.\ the following diagrams commute for each morphism $\mu\colon m\to n$ in $\pp$ and objects $c,c'\in\cat{C}$:
\[
\begin{tikzcd}[column sep=50pt]
	(Fc\otimes Fc')\tpow{m}
		\ar[r, "t_{Fc\otimes Fc'}(\mu)"]\ar[d, "(\varphi_{c,c'})\tpow{m}"']&
 	(Fc\otimes Fc')\tpow{n}\ar[d, "(\varphi_{c,c'})\tpow{n}"]\\
  F(c\otimes c')\tpow{m}\ar[r, "t_{F(c\otimes c')}(\mu)"']&
  F(c\otimes c')\tpow{n}
\end{tikzcd}
\hspace{.7in}
\begin{tikzcd}[column sep=35pt]
	I\tpow{m}
		\ar[r, "t_I(\mu)"]\ar[d, "\varphi\tpow{m}"']&
 	I\tpow{n}\ar[d, "\varphi\tpow{n}"]\\
  F(I)\tpow{m}\ar[r, "t_{F(I)}(\mu)"']&
  F(I)\tpow{n}
\end{tikzcd}
\]
\end{theorem}
\begin{proof}
Each of these is proved by a diagram chase. Indeed, consider the diagram:
\[
\begin{tikzcd}[column sep=6pt]
	(Fc\otimes Fc')\tpow{m}
		\ar[rrrr, "t_{Fc\otimes Fc'}(\mu)"]
		\ar[ddddd, "\varphi\tpow{m}"']
		\ar[from=dr, "\sigma"]&&[18pt]
		\ar[d, phantom, near start,  "\textsc{\tiny (t)}"]&[18pt]&
	(Fc\otimes Fc')\tpow{n}
		\ar[ddddd, "\varphi\tpow{n}"]
		\ar[from=dl, "\sigma"']\\&
	(Fc)\tpow{m}\otimes(Fc')\tpow{m}
		\ar[d, "\varphi"']
		\ar[rr, "t_{Fc}(\mu)\otimes t_{Fc'}(\mu)"]&
		\ar[d, phantom, near start, "\textsc{\tiny (Fp)}"]&
	(Fc)\tpow{n}\otimes(Fc')\tpow{n}
		\ar[d, "\varphi"]\\&
	F(c\tpow{m})\otimes F(c'\,\tpow{m})
		\ar[rr, "Fs_c(\mu)\otimes Fs_{c'}(\mu)"]
		\ar[d, "\varphi"']&
		\ar[d, phantom, near start, "\textsc{\tiny (Fm)}"]&
	F(c\tpow{n})\otimes F(c'\,\tpow{n})
		\ar[d, "\varphi"]\\&
	F(c\tpow{m}\otimes c'\,\tpow{m})
		\ar[rr, "F(s_c(\mu)\otimes s_{c'}(\mu))"]&
		\ar[d, phantom, near start, "\textsc{\tiny (s)}"]&
	F(c\tpow{n}\otimes c'\,\tpow{n})\\&
	F((c\otimes c')\tpow{m})
		\ar[u, "\sigma"]
		\ar[rr, "F(s_{c\otimes c'}(\mu))"]&
		\ar[d, phantom, near start, "\textsc{\tiny (Fp)}"]&
	F((c\otimes c')\tpow{n})
		\ar[u, "\sigma"']\\
	F(c\otimes c')\tpow{m}
		\ar[rrrr, "t_{F(c\otimes c')}(\mu)"]
		\ar[ur, "\varphi"]&&~&&
	F(c\otimes c')\tpow{n}
		\ar[ul, "\varphi"']
\end{tikzcd}
\]
Every vertical or diagonal morphism is an isomorphism, and the (unlabeled) side diagrams commute because symmetries commute with strongators. Diagrams \textsc{(s)} and \textsc{(t)} commute because $s$ and $t$ are supplies (see \cref{eqn.supply_commute_tensors}); diagrams \textsc{(Fp)} commute because $F$ preserves the supply (see \cref{eqn.unpack_preserve_supply}); and \textsc{(Fm)} commutes because $F$ is monoidal.

The proof for the unit is similar, except that squares \textsc{(Fp)} are not present.
\end{proof}

Recall from \cref{thm.homos_form_subcat} that for any supply $s$ in $\cat{C}$ there is a symmetric monoidal subcategory $\cat{C}_s\ss\cat{C}$ of $s$-homomorphisms.

\begin{proposition}\label{prop.preserve_homs}
If $F$ preserves supply, it sends homomorphisms to homomorphisms, i.e.\ it restricts to a strong monoidal functor $F_{s,t}\colon\cat{C}_s\to\cat{D}_t$.
\end{proposition}
\begin{proof}
Choose $\mu\colon m\to n$ in $\pp$ and $f\colon c\to d$ in $\cat{C}$, and consider the diagram below:
\[
\begin{tikzcd}[column sep=6pt]
	(Fc)\tpow{m}
		\ar[rrrr, "t_{Fc}(\mu)"]
		\ar[dddd, "(Ff)\tpow{m}"']
		\ar[dr, "\varphi"]
		&
    &[23pt] 
    \ar[d, phantom, near start,  "\textsc{\tiny (Fp)}"]
		&[23pt]
    &
  (Fc)\tpow{n}
		\ar[dddd, "(Ff)\tpow{n}"]
		\ar[dl, "\varphi"']
		\\
		&
	F(c\tpow{m})
		\ar[dd, "F(f\tpow{m})"]
    \ar[rr, "F(s_{c}(\mu))"]
    &
    ~
    &
	F(c\tpow{n})
    \ar[dd, "F(f\tpow{n})"']
    \\
    \ar[r, phantom,"\textsc{\tiny (Fm)}"]
    &~& 
    \textsc{\tiny (s)}
    &~&
    \ar[l, phantom,"\textsc{\tiny (Fm)}"]
		\\
		&
	F(d\tpow{m})
		\ar[rr, "F(s_d(\mu))"']
    \ar[dl, "\varphi"']
    &
    \ar[d, phantom, pos=.6, "\textsc{\tiny (Fp)}"]
    &
	F(d\tpow{n})
    \ar[dr, "\varphi"]
    \\
	(Fd)\tpow{m}
		\ar[rrrr, "t_{Fd}(\mu)"']
		&&~&&
	(Fd)\tpow{n}
\end{tikzcd}
\]
The diagrams \textsc{(Fp)} commute because $F$ preserves supply, while the diagrams \textsc{(Fm)} commute because $F$ is monoidal. Thus whenever $f$ is an $s$-homomorphism, the functoriality of $F$ implies \textsc{(s)} commutes, and hence that $F(f)$ is a $t$-homomorphism.
\end{proof}

%\begin{proposition}\label{prop.easy_pres_supply}
%Let $s$ and $t$ be supplies of $\pp$ in $\cat{C}$ and $\cat{D}$ respectively, and suppose that $F\colon\cat{C}\to\cat{D}$ sends $s$-homomorphisms to $t$-homomorphisms, i.e.\ $F$ factors through $F_{s,t}\colon\cat{C}_s\to\cat{D}_t$. Then $F$ preserves the supply iff the strongators $\varphi_c\colon F(c)\tpow{m}\to F(c\tpow{m})$ define a natural isomorphism:
%\begin{equation}\label{eqn.pres_supply}
%\begin{tikzcd}[column sep=55pt]
%	\pp\ar[r, "s"]\ar[d, "t"']&
%	\smf(\cat{C}_s,\cat{C})\ar[d, "{\smf(\cat{C}_s,F)}"]\\
%	\smf(\cat{D}_t,\cat{D})\ar[r, "{\smf(F_{s,t},\cat{D})}"']&
%	\smf(\cat{C}_s,\cat{D})\ar[ul, phantom, "\overset{\varphi}{\cong}"]
%\end{tikzcd}
%\end{equation}
%\end{proposition}
%\begin{proof}
%Consider an object $m\in\pp$. Along the top-right, it is sent to the functor $c\mapsto F(c\tpow{m})$, and along the left-bottom, it is sent to the functor $c\mapsto F(c)\tpow{m}$. The strongators for $F$ provide the component isomorphisms $\varphi_c\colon F(c)\tpow{m}\to F(c\tpow{m})$ natural in $c\in\cat{C}$ (and hence in $c\in\cat{C}_s$). For these $\varphi$ to be natural in $\pp$ works out to be exactly the condition that  \cref{eqn.unpack_preserve_supply} commutes for any morphism $\mu\colon m\to n$ in $\pp$.
%\end{proof}

\begin{proposition}\label{prop.easy_pres_supply}
Let $s$ and $t$ be supplies of $\pp$ in $\cat{C}$ and $\cat{D}$ respectively. Then $F\colon\cat{C}\to\cat{D}$ preserves the supply iff
\begin{itemize}
	\item $F$ sends coherence maps to $t$-homomorphisms, i.e.\ it factors through some $F_{0,t}\colon\cat{C}_0\to\cat{D}_t$
	\item the strongators $\varphi_c\colon F(c)\tpow{m}\to F(c\tpow{m})$ define a natural isomorphism:
\begin{equation}\label{eqn.pres_supply}
\begin{tikzcd}[column sep=55pt]
	\pp\ar[r, "s"]\ar[d, "t"']&
	\smf(\cat{C}_0,\cat{C})\ar[d, "{\smf(\cat{C}_0,F)}"]\\
	\smf(\cat{D}_t,\cat{D})\ar[r, "{\smf(F_{0,t},\cat{D})}"']&
	\smf(\cat{C}_0,\cat{D})\ar[ul, phantom, "\overset{\varphi}{\cong}"]
\end{tikzcd}
\end{equation}
\end{itemize}
\end{proposition}
\begin{proof}
By \cref{prop.preserve_homs}, if $F$ preserves the supply then it sends coherence maps (in fact all $s$-homomorphisms) to $t$-homomorphisms. So we may assume we have $F_{0,t}$ and prove that $F$ preserves the supply iff diagram \eqref{eqn.pres_supply} commutes.

Consider an object $m\in\pp$. Along the top-right, it is sent to the functor $c\mapsto F(c\tpow{m})$, and along the left-bottom, it is sent to the functor $c\mapsto F(c)\tpow{m}$. The strongators for $F$ provide the component isomorphisms $\varphi_c\colon F(c)\tpow{m}\to F(c\tpow{m})$ natural in $c\in\cat{C}$ (and hence in $c\in\cat{C}_0$). For these $\varphi$ to be natural in $\pp$ works out to be exactly the condition that  \cref{eqn.unpack_preserve_supply} commutes for any morphism $\mu\colon m\to n$ in $\pp$.
\end{proof}

\cref{cor.strict_equiv_pres_supply} follows easily from the proof of \cref{thm.supply_strictification}.

\begin{corollary}\label{cor.strict_equiv_pres_supply}
Let $s$ be a supply of $\pp$ in $\cat{C}$ and let $\strict{s}$ be the induced supply of $\pp$ in the strictification $\strict{\cat{C}}$. Then the equivalence $\bigotimes\colon\strict{\cat{C}}\to\cat{C}$ preserves the supply.
\end{corollary}

%======== Chapter ========%
\chapter{Outlook}

Many of the ideas in this paper should extend to the enriched setting, e.g.\ replacing props and symmetric monoidal categories with 2-props and symmetric monoidal 2-categories, etc. Indeed, in \cite{fong2019abelian}, we work out the theory for the locally posetal case. The results contained here, and their locally posetal generalizations, organize and significantly streamline key arguments in that paper.

We leave the development of the general enriched theory open for future work.

%============ Reference ============%

\printbibliography

\newpage
%============ Appendix ============%
\appendix

%======== Chapter ========%
\chapter{Products, coproducts, and biproducts in $\mathbb{S}\mathsf{MC}$}\label{chap.proofs}

In a category with products, we denote the pairing of $f\colon A\to B$ and $g\colon A\to C$ by $\pair{f,g}\colon A\to B\times C$. We will denote copairings by $\copair{-,-}$.

\begin{theorem*}[\ref{thm.smc_biprod}]\label{page.smc_biprod}
The 2-category $\ssmc$ of symmetric monoidal categories, strong monoidal functors, and monoidal natural transformations has 2-categorical biproducts.
\end{theorem*}
\begin{proof}
The terminal category $\zero\coloneqq\{*\}$ is symmetric monoidal, and it is terminal as such. It is also 2-categorically initial: for every monoidal category $(\cat{C},I,\otimes)$, the functor $I\colon\zero\to\cat{C}$ sending $*\mapsto I$ is strong monoidal and any other strong monoidal functor $\zero\to\cat{C}$ is canonically isomorphic to $I$. Thus $\cat{I}$ is a 2-categorically a zero object.

Let $\cat{C}$ and $\cat{D}$ be symmetric monoidal categories. Their product $\cat{C}\times\cat{D}$ as categories inherits a symmetric monoidal structure. Indeed, take $(I,I)$ to be the monoidal unit and $(c_1,d_1)\otimes(c_2,d_2)\coloneqq(c_1\otimes c_2,d_1\otimes d_2)$ to be the monoidal product; the associators, unitors, and braiding are given pointwise. We will denote this symmetric monoidal category by $\cat{C}\oplus\cat{D}\coloneqq\cat{C}\times\cat{D}$ and show that it is a biproduct; more precisely it is both a 2-categorical coproduct and a strict 2-categorical product.

The functor $\pair{\cat{C},I}\colon\cat{C}\to\cat{C}\oplus\cat{D}$ sending $c\mapsto (c,I)$ is clearly strong monoidal. We claim that it and $\pair{I,\cat{D}}\colon\cat{D}\to\cat{C}\oplus\cat{D}$ together form the coprojections under which $\cat{C}\oplus\cat{D}$ is a 2-categorical coproduct. Indeed, given strong monoidal functors $F\colon\cat{C}\to\cat{X}$ and $G\colon\cat{D}\to\cat{X}$, define their copairing $\copair{F,G}\colon\cat{C}\oplus\cat{D}\to\cat{X}$ by $\copair{F,G}(c,d)\coloneqq F(c)\otimes G(d)$, and similarly for morphisms. The result is strong monoidal: as strongator we take the composite
\begin{align*}
	\copair{F,G}(c_1,d_1)\otimes\copair{F,G}(c_2,d_2)&=
	F(c_1)\otimes G(d_1)\otimes F(c_2)\otimes G(d_2)\\&\cong
	F(c_1)\otimes F(c_2)\otimes G(d_1)\otimes G(d_2)\\&\cong
	\copair{F,G}\big((c_1,d_1)\otimes(c_2,d_2)\big),
\end{align*}
where the first isomorphism is the braiding in $\cat{C}$ and the second isomorphism uses the strongators from $F$ and $G$. It is straightforward to check that this satisfies the necessary axioms to be a strongator%
\hide[.]{
, e.g.\ the commutativity of the following diagram
\[
\begin{tikzcd}
	Fc_1\otimes Gd_1\otimes Fc_2\otimes Gd_2\otimes Fc_3\otimes Gd_3\ar[r, "\sigma"]\ar[d, "\sigma"']&
	Fc_1\otimes Gd_1\otimes Fc_2\otimes Fc_3\otimes Gd_2\otimes Gd_3\ar[d]\\
	Fc_1\otimes Fc_2\otimes Gc_1\otimes Gc_2\otimes Fc_3\otimes Gd_3\ar[d]&
	Fc_1\otimes Gd_1\otimes F(c_2\otimes c_3)\otimes G(d_2\otimes d_3)\ar[d, "\sigma"]\\
	F(c_1\otimes c_2)\otimes G(d_1\otimes d_2)\otimes Fc_3\otimes Gd_3\ar[d, "\sigma"']&
	Fc_1\otimes F(c_2\otimes c_3)\otimes Gd_1\otimes G(d_2\otimes d_3)\ar[d]\\
	F(c_1\otimes c_2)\otimes Fc_3\otimes G(d_1\otimes d_2)\otimes Gd_3\ar[r]&
	F(c_1\otimes c_2\otimes c_3)\otimes G(d_1\otimes d_2\otimes d_3)	
\end{tikzcd}
\]
}
It is also easy to check that the unitors provide natural isomorphisms
\begin{equation}\label{eqn.coproduct_smc}
\begin{tikzcd}[sep=large]
	\cat{C}\ar[r, "\pair{\cat{C},I}"]\ar[dr, bend right=20pt, "F"', "" name=F]&
	|[alias=CD]|\cat{C}\oplus\cat{D}\ar[d, "\copair{F,G}" description]&
	\cat{D}\ar[l, "\pair{I,\cat{D}}"']\ar[dl, bend left=20pt, "G", ""' name=G]\\&
	\cat{X}
	\ar[from=G, to=CD, phantom, near start, "\cong"]
	\ar[from=F, to=CD, phantom, near start, "\cong"]
\end{tikzcd}
\end{equation}
e.g.\ $c\otimes I\cong c$ for any $c\in\cat{C}$. The map $\copair{F,G}$ is determined (up to canonical isomorphism) by this property because every object in $\cat{C}\oplus\cat{D}$ is of the form $(c,I)\otimes(I,d)$, and similarly for morphisms. Thus we have established that $\cat{C}\oplus\cat{D}$ is a 2-categorical coproduct.

We claim it is also the (strict) product using the usual projections, e.g.\ $\pi_{\cat{C}}\colon\cat{C}\times\cat{D}\to\cat{C}$. These functors are easily seen to be strong monoidal. Given any symmetric monoidal category $\cat{X}$ and functors $F\colon\cat{X}\to\cat{C}$ and $G\colon\cat{X}\to\cat{D}$, we get a universal functor $\pair{F,G}\colon\cat{X}\to\cat{C}\times\cat{D}$; we need to see that if $F$ and $G$ are strong monoidal then so is $\pair{F,G}$. Indeed we have
\begin{align*}
	\pair{F,G}(x_1)\otimes\pair{F,G}(x_2)&=
	\big(F(x_1),G(x_1)\big)\otimes\big(F(x_2),G(x_2)\big)\\&=
	\big(F(x_1)\otimes F(x_2),G(x_1)\otimes G(x_2)\big)\\&\cong
	\big(F(x_1\otimes x_2),G(x_1\otimes x_2)\big)\\&=
	\pair{F,G}(x_1\otimes x_2).
\end{align*}
The product universal property diagram analogous to \cref{eqn.coproduct_smc} commutes (on the nose), completing the proof that $\ssmc$ has biproducts.
\end{proof}

\begin{proposition*}[\ref{prop.biprod_smf_strict}]\label{page.biprod_smf_strict}
Let $\cat{C}_1,\cat{C}_2,\cat{D}_1,\cat{D}_2$ be symmetric monoidal categories. The functor
\begin{equation}\label{eqn.strict_smf_biprod}\oplus\colon\smf(\cat{C}_1,\cat{D}_1)\times\smf(\cat{C}_2,\cat{D}_2)\to\smf(\cat{C}_1\oplus\cat{C}_2,\cat{D}_1\oplus\cat{D}_2),
\end{equation}
given by $(F_1\oplus F_2)(c_1, c_2)\coloneqq (F_1(c_1),F_2(c_2))$, is strict monoidal.
\end{proposition*}
\begin{proof}
The monoidal unit in the domain is the pair $(I,I)$ of constant functors, and it is clearly sent to the monoidal unit $(I,I)$ in the codomain. Thus $\oplus$ commutes with the monoidal unit; we need to check that it commutes with the monoidal product $\otimes$.

Suppose given $F_1, F_1'\colon\cat{C}_1\to\cat{D}_1$ and $F_2,F_2'\colon\cat{C}_2\to\cat{D}_2$. Then we have equalities
\begin{align*}
	\big((F_1\oplus F_2)\otimes(F_1'\oplus F_2')\big)(c_1,c_2)&=
	(F_1\oplus F_2)(c_1,c_2)\otimes(F_1'\oplus F_2')(c_1,c_2)\\&=
	\big(F_1(c_1),F_2(c_2)\big)\otimes\big(F_1'(c_1),F_2'(c_2)\big)\\&=
	\big(F_1(c_1)\otimes F_1'(c_1),F_2(c_2)\otimes F_2'(c_2)\big)\\&=
	\big((F_1\otimes F_1')\oplus(F_2\otimes F_2')\big)(c_1,c_2)
\end{align*}
for any $c_1\in\cat{C}_1$ and $c_2\in\cat{C}_2$. This establishes strictness, and a similar calculation implies that $\oplus$ preserves the braiding.
\end{proof}

\begin{theorem*}[\ref{thm.smc_all_prod_coprod}]\label{page.smc_all_prod_coprod}
The 2-category $\ssmc$ has all small products and coproducts, and products are strict.
\end{theorem*}
\begin{proof}[Sketch of proof]
Let $J$ be a set and $\cat{C}_\bullet\colon J\to\ssmc$ be an $J$-indexed collection of symmetric monoidal categories. Their product as categories $\prod_{j\in J}\cat{C}_j$ carries a symmetric monoidal structure given elementwise on $J$. It is easy to check that this, together with the usual projections (which are strict monoidal functors), constitutes the product of the $\cat{C}_j$ in $\ssmc$.

The coproduct $\bigsqcup_{j\in J}\cat{C}_j$ has the following set of objects, where $I_j$ is the unit in $\cat{C}_j$:
\[
\ob\bigg(\bigsqcup_{j\in J}\cat{C}_j\bigg)\coloneqq
\bigg\{
c\in\prod_{j\in J}\ob(\cat{C}_j)
\;\bigg|\;
c_j=I_j\text{ for all but finitely-many }j\in J
\bigg\}.
\]
The monoidal product is given pointwise (and then replace $I\otimes I$ by $I$). We leave to the reader to check that this, together with the obvious coprojections $\inc_j\colon\cat{C}_j\to\bigsqcup_{j\in J}\cat{C}_j$, constitutes a (2-categorical) coproduct in $\ssmc$, i.e.\ that for any symmetric monoidal category $\cat{X}$, there is an equivalence of categories
\[
  \ssmc\Big(\bigsqcup_{j\in J}\cat{C}_j,\cat{X}\Big)\simeq\prod_{j\in J}\ssmc(\cat{C}_j,\cat{X}).
\qedhere
\]
\end{proof}

\end{document}